\documentclass[a4paper,11pt, reqno]{amsart}

\usepackage[margin=2cm]{geometry}
\usepackage[authoryear]{natbib}

\usepackage[hidelinks,colorlinks = true,linkcolor={blue!80!black},citecolor={blue!50!black},urlcolor={blue!80!black}]{hyperref}
\usepackage{amsmath,amsfonts}
\usepackage{amsthm}
\usepackage{enumitem}
\usepackage{amssymb}

\usepackage{multirow}
\usepackage{float}
\usepackage{subcaption}
\usepackage{color}
\usepackage{soul}
\usepackage{algpseudocode}
\usepackage{algorithm}

\algnewcommand{\IfThenElse}[3]{
  \State \algorithmicif\ #1\ \algorithmicthen\ #2\ \algorithmicelse\ #3}

\usepackage[colorinlistoftodos]{todonotes}

\def\R{\mathbb{R}}

\def\B{\mathcal{B}}
\def\T{\mathcal{T}}
\def\J{\mathcal{J}}
\def\I{\mathcal{I}}
\def\K{\mathcal{K}}
\def\S{\mathcal{S}}

\newcommand{\dsum}{\displaystyle\sum}

\newtheorem{theorem}{Theorem}

\newtheorem{prop}[theorem]{Proposition}
\newtheorem{example}{Example}

\allowdisplaybreaks

\usepackage{adjustbox}
\usepackage{booktabs}
\usepackage{multirow}

\allowdisplaybreaks

\let\origmaketitle\maketitle
\def\maketitle{
	\begingroup
	\def\uppercasenonmath##1{} 
	\let\MakeUppercase\relax 
	\origmaketitle
	\endgroup
}

\begin{document}
	
	\title[]{\Large The Cooperative Maximal Covering Location Problem with \\ ordered partial attractions}
	
\author[C. Domínguez, R. G\'azquez, J.M. Morales,  \MakeLowercase{and} S. Pineda]{
{\large Concepci\'on Dom\'inguez$^\dagger$, Ricardo G\'azquez$^\ddagger$, Juan Miguel Morales$^{\dagger\dagger}$, and Salvador Pineda$^{\ddagger\ddagger}$}\medskip\\
$^\dagger$Department of Statistics and Operational Research, Universidad de Murcia, Murcia, Spain.\\
$^\ddagger$Department of Statistics, Universidad Carlos III de Madrid, Madrid, Spain.\\
$^{\dagger\dagger}$Department of Mathematical Analysis, Statistics and Operations Research \& Applied Mathematics, Universidad de Málaga, Málaga, Spain .\\
$^{\ddagger\ddagger}$Department of Electrical Engineering, Universidad de Málaga, Málaga, Spain.\\
\texttt{concepcion.dominguez@um.es}, \texttt{ricardo.gazquez@uc3m.es}, \texttt{juan.morales@uma.es}, \texttt{spineda@uma.es}
}
	
\date{\today}

\maketitle 	

\begin{abstract}
The Maximal Covering Location Problem (MCLP) is a classical location problem where a company maximizes the demand covered by placing a given number of facilities, and each demand node is covered if the closest facility is within a predetermined radius. In the cooperative version of the problem (CMCLP), it is assumed that the facilities of the decision maker act cooperatively to increase the customers’ attraction towards the company. In this sense, a demand node is covered if the aggregated partial attractions (or partial coverings) of open facilities exceed a threshold. In this work, we generalize the CMCLP introducing an Ordered Median function (OMf), a function that assigns importance weights to the sorted partial attractions of each customer and then aggregates the weighted attractions to provide the total level of attraction. We name this problem the \emph{Ordered Cooperative Maximum Covering Location Problem} (OCMCLP). The OMf serves as a means to compute the total attraction of each customer to the company as an aggregation of ordered partial attractions and constitutes a unifying framework for CMCLP models.

We introduce a multiperiod stochastic non-linear formulation for the CMCLP with an embedded assignment problem characterizing the ordered cooperative covering. For this model, two exact solution approaches are presented: a MILP reformulation with valid inequalities and an effective approach based on Generalized Benders' cuts. Extensive computational experiments are provided to test our results with randomly generated data and the problem is illustrated with a case study of locating charging stations for electric vehicles in the city of Trois-Rivières, Québec (Canada).
\end{abstract}
		
\keywords{Maximal Covering; Cooperative Cover; Facility Location; Ordered Median Function; Mixed-integer Optimization; Benders Decomposition.}

\section{Introduction}
Covering problems constitute a fruitful research area within the facility location field. A demand point or customer class is covered if a facility is installed within a fixed covering radius. The aim is thus to maximize the covered demand minimizing the number of facilities installed. Two related problems arise: the Maximal Covering Location Problem (MCLP), introduced by \cite{church1974}, where the objective is to maximize the demand covered with a limited number of facilities; and the Set Covering Location Problem (SCLP) introduced by \cite{toregas1971}, that minimizes the number of facilities needed to cover all the demand. A third problem, the $p$-center location problem, arises when the radius is not predetermined, but rather a variable of the model.

In these classical problems, the covering is determined by one facility, namely the closest to the demand node. This is known as \textit{individual} covering. However, it is not hard to think of general settings in which the cooperation of facilities to provide covering is more suitable. For instance, in emergency response facility location \citep[see][]{hogan1986, daskin1988, li2011} some backup installations for emergency services are necessary to consider a demand node as covered. In retailing applications, the  distance and number of facilities installed is crucial to determine the level of covering of a customer class (in this setting, the cooperation is frequently seen as covering a fraction of the customer's demand). Since the facilities cooperate to provide coverage, these generalizations fall in the category of \textit{cooperative covering}. This covering has been studied from different points of view, depending on the assumptions made on the facilities and the customers, leading to a wide range of problems that fit many different applications.

The cooperative covering in the plane was formally introduced by \cite{berman2009}. The authors consider the setting where each facility emits a (physical or non-physical) \emph{signal} that decays over the distance, and each demand node is covered if its total (aggregated) signal exceeds a predefined threshold. For physical signals (such as light, sound, microwave signals, phone coverage), in many cases the signal emitted by the source dissipates proportionally to the square of the travel distance. They give the cooperative counterparts of all three of the classical models previously stated, suggesting optimal algorithms for the problem with two facilities and heuristics for more than two facilities. Their solutions are illustrated with a case study of locating warning sirens in North Orange County, California. Other applications described for physical signals are the location of cell phone towers, light towers to provide adequate lighting to an area, and the placement of outdoor gas heaters in cafe’s or restaurant's terraces. 

The cooperative cover framework for the discrete location case (where the set of potential facility locations is discrete) is analyzed in \cite{berman2011}. Again, extensions for the three classical covering problems are introduced, resulting in the Cooperative Maximal Covering Location Problem (CMCLP), the Cooperative Set Covering Location Problem and the Generalized Cooperative $p$-center. The assumptions made on the covering mechanism are as in \cite{berman2009}, and both exact and algorithmic approaches are proposed.

\cite{averbakh2014} study the CMCLP on a network, allowing the facilities to be located at both the nodes and along the edges. This setting is more appropriate for applications with non-physical signals, because they do not generally propagate along the straight-line distance. Applications include customers that need to be served by a facility in a service time in the retail and take-out food, and they consider a customer to be covered if the probability that there is a facility that can serve it in time is sufficiently high. For the CMCLP with two facilities and a linear signal strength function, they present an $O(m^2n^2)$ exact algorithm (for $n$ nodes and $m$ edges). For the general case, they propose a mixed-integer programming (MIP) formulation that can be used for small instances, and greedy-type and local search-based interchange heuristics for large instances. More recently, \cite{li2018} propose a cooperative model for humanitarian relief management, \cite{liu2022} consider cooperative covering under uncertain demand, and \citet{baldassarre2024two} propose a cooperative covering problem which was applied to the bank sector.

The subject of this article is a generalization of the CMCLP in the discrete setting that constitutes a unified modeling framework for the non-cooperative and a wide range of cooperative settings. To do so, we extend the definition of cooperative covering using an ordered median function (OMf). Thus, the total attraction of a demand node towards the company is obtained by (1) ordering the partial attractions provided by the open facilities from largest to smallest, and (2) summing up the partial attractions weighted by a vector of importance weights $\lambda$. Finally, the node is covered if the total attraction exceeds a threshold. The OMf is a very general function that has as particular cases the classical MCLP (considering $\lambda_1=1$ and $\lambda_j=0$ for the last components, the covering is given by the most attractive facility, where the indices $j$ refer to the ordered facilities) or the standard cooperative setting given by the aggregated sum of partial attractions (by choosing weights $\lambda_j = 1$ for all $j$), so it generalizes the model proposed in \cite{berman2011}. Finally, it extends models like the one introduced in \cite{lin2021} in the context of the Maximal Capture Location Problem with discrete choice rule (where the demand of the customer is entirely fulfilled by the company that provides the highest coverage). In this work, the customers form a \emph{consideration set} containing the $\ell$ most attractive open facilities to determine the total coverage. This consideration set can be modeled by choosing $\lambda_j=1$ for the first $\ell$ components and $\lambda_j=0$ for the rest. We name the problem the Ordered Cooperative Maximal Covering Location Problem (OCMCLP), and, to the best of our knowledge, no ordered cooperative coverings have been addressed in the literature so far.

In the context of Location Theory with the OMf, there is a vast literature under the umbrella of the Ordered Median Problem \citep{puertofdez94, Puerto2019}, since the most used objective functions (e.g., median, center, $k$-centrum or centian) can be covered by OMfs. Thus, it has been successfully applied to a wide range of areas such as hub location problems \citep{puerto2011single,puerto2016ordered} or $p$-median problems both discrete \citep{deleplanque2020branch,marin2020fresh} and continuous \citep{BEP16,blanco2023branch}. In fact, this generalization has already been tackled in the context of a related type of covering location problem: the gradual covering. In gradual covering, the \emph{all-or-nothing coverage} assumption is replaced with a general coverage function, generally non-increasing with distance, which represents the proportion of demand covered by the facility. For instance, two radii $r_\ell < r_u$ can be considered so that: (1) customers who are closer than $r_\ell$ to their closest facility are fully covered, (2) customers who are further away than $r_u$ to their closest facility are not covered, and (3) customers with travel distance in between $r_\ell$ and $r_u$ are partially covered. This problem has been extensively studied, sometimes under the names of partial covering or generalized covering \citep{berman2002, berman2003,drezner2010}. The Ordered Gradual Covering Location Problem (OGCLP) combines the characteristics of ordered median and gradual cover and was introduced in \citet{berman2009b}. Note that the gradual cover is not cooperative, since it is determined solely by the closest facility, and that in the OGCLP the \emph{customers} are ordered according to their levels of covering, whereas in our setting, the partial attractions provided by the \emph{facilities} are ordered for each customer to form their total cooperative coverage. This translates to the OMf being part of the objective function in the OGCLP, whereas in our problem, the OCMCLP, the addition of the OMf leads to an embedded assignment problem. To further emphasize the distinction between gradual and cooperative covering, we illustrate the gradual, cooperative and gradual cooperative coverings in Figure \ref{fig:example_coop} and refer the reader to \cite{berman2010}. Recent examples of problems with gradual cooperative covering are given in \citet{karatas2017,bagherinejad2018, karatas2021, usefi2023}. Other types of covering problems have also been generalized using the OMf \citep{BP21,blanco2022fairness}.

\begin{figure}
    \centering
    \begin{subfigure}[b]{.32\linewidth}
    \centering
    \fbox{\includegraphics[scale = 0.18]{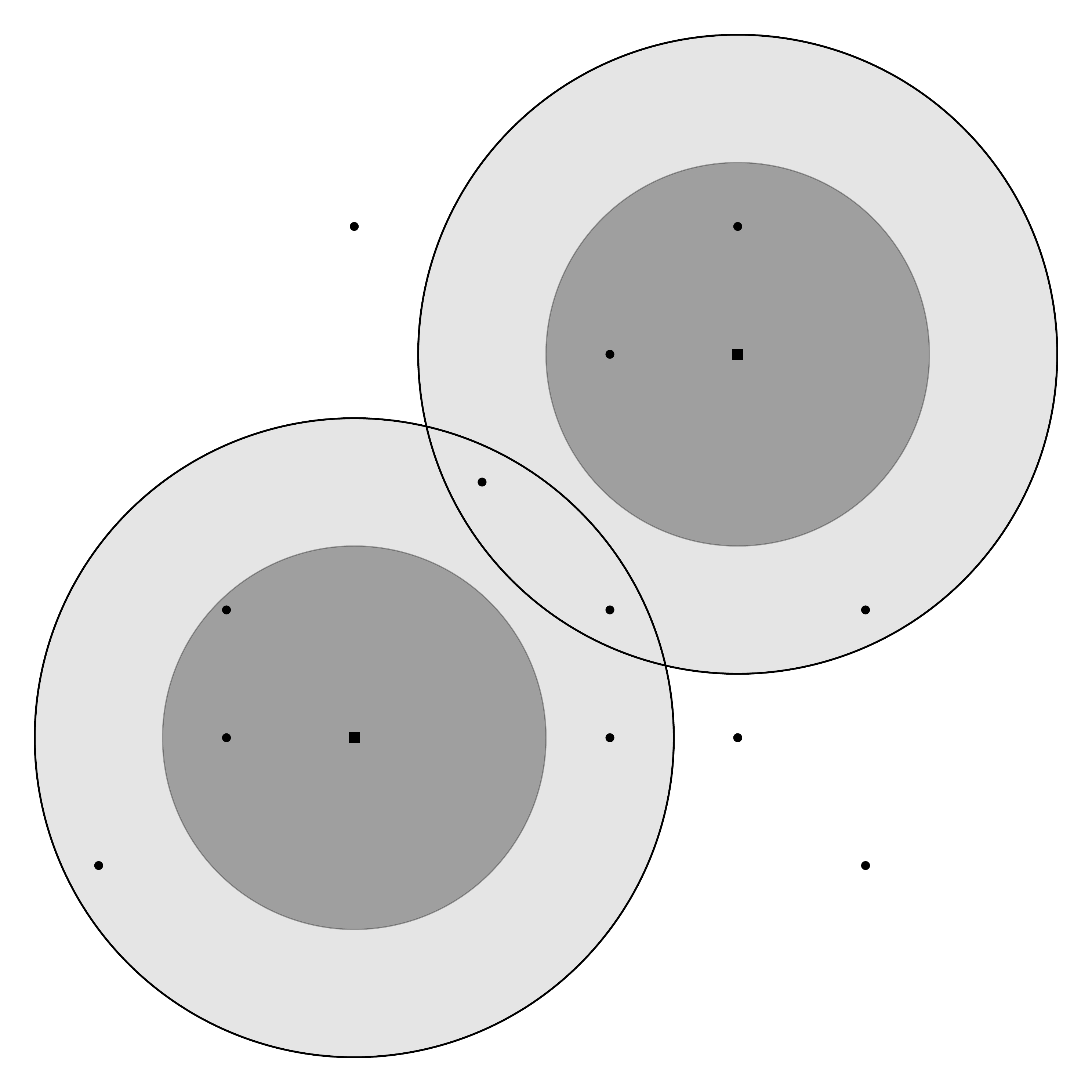}}
    \caption{Gradual covering.}\label{fig:ex1gradcov}
    \end{subfigure}~
    \begin{subfigure}[b]{.32\linewidth}
    \centering 
    \fbox{\includegraphics[scale = 0.18]{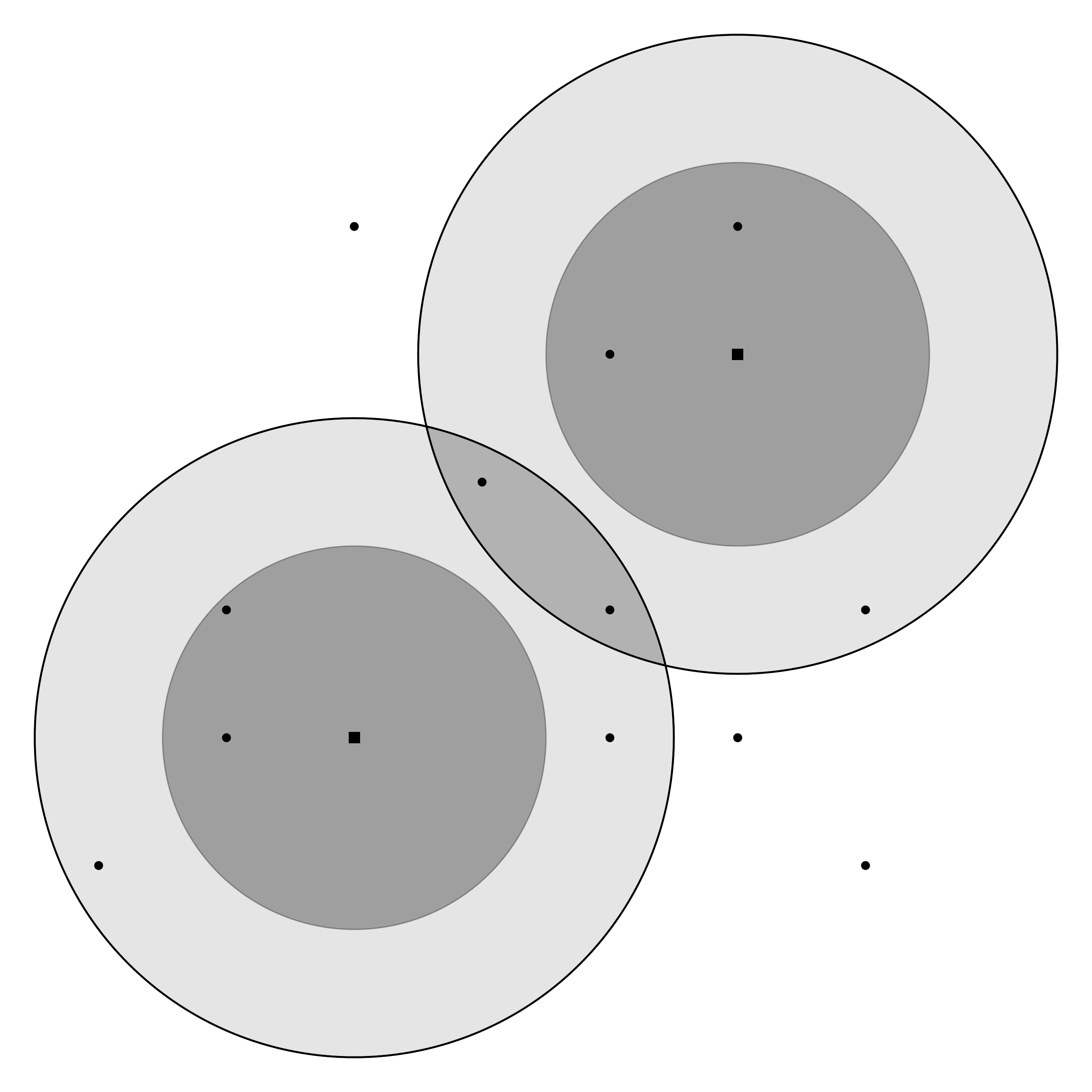}}
    \caption{Gradual cooperative covering.}\label{fig:ex1coopgradcov}
    \end{subfigure}~
    \begin{subfigure}[b]{.32\linewidth}
    \centering 
    \fbox{\includegraphics[scale = 0.18]{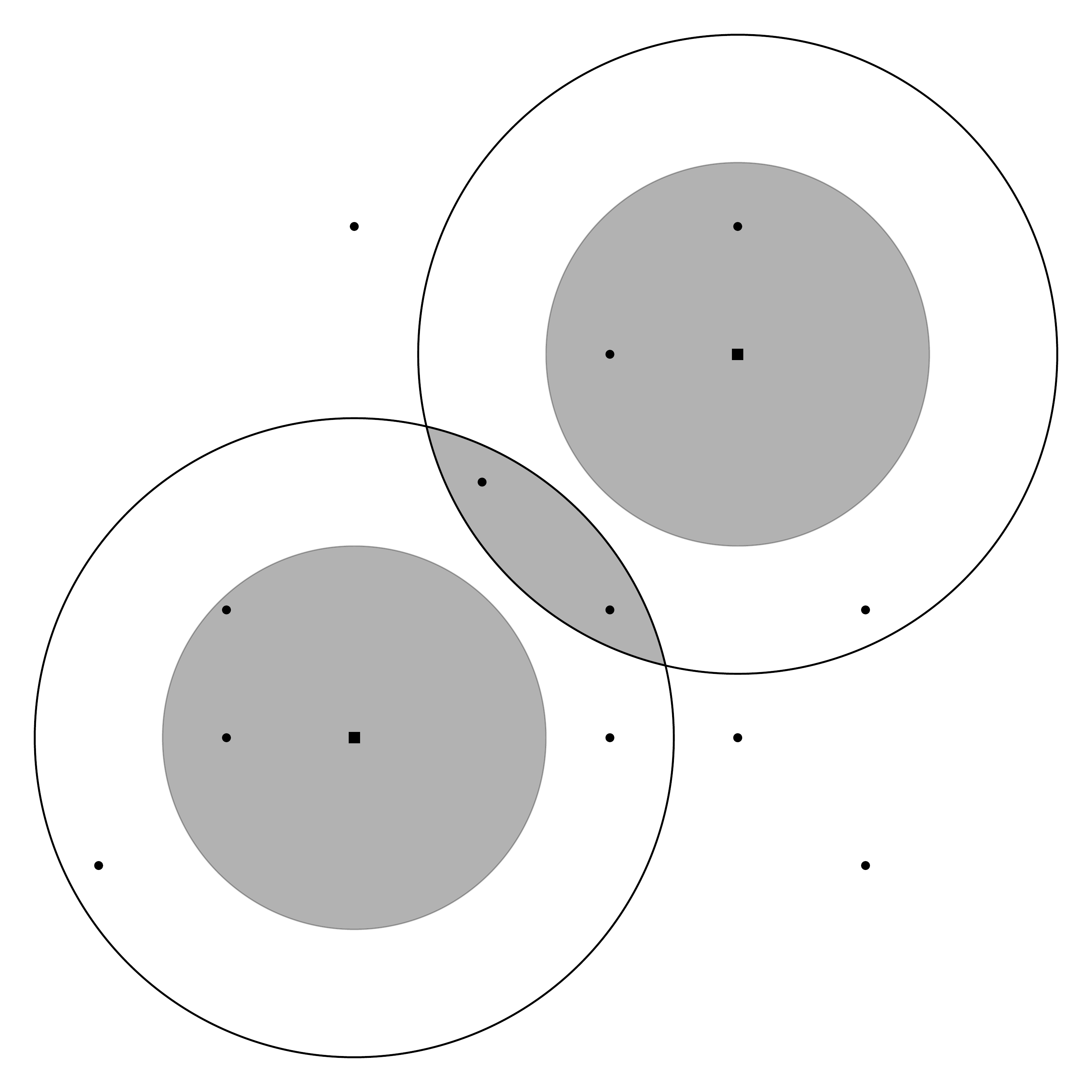}}
    \caption{Cooperative covering.}\label{fig:coopcov}
    \end{subfigure}
    \caption{Illustrative example of gradual and cooperative coverings. Circles denote clients, and squares the two located facilities. The clients located in the white area are not covered, the light gray area corresponds to partially covered, and the dark gray area includes fully covered clients.}
    \label{fig:example_coop}
\end{figure}

The combination of the discrete location and the cooperative (non-gradual) setting suits nearly all applications in the physical signal context, as well as many applications in the non-physical signal category. In the latter, \cite{berman2009} point out that the signal can be seen as an attraction of the facility (i.e.\ a \emph{utility} associated to the value of a given alternative), so we follow their approach and consider partial attractions of the customer towards each facility. In their setting, this attraction depends on the distance. In our case, the attraction is generalized and depends on features such as the distance to the facility, its size or its location near relevant establishments. Because it also relies on the perception of the customers, we consider this partial attractions as stochastic. 
Therefore, if the decision-maker places two facilities, it makes sense to assume that a customer's attraction to a particular facility may decrease if other establishments are installed closer, thus the OMf with non-increasing weights perfectly models this setting. Since gradual covering is not considered, the applications in retail are appropriate when the use of the facilities of a company requires a previous affiliation or a specific product, making it impractical or impossible to patronize facilities from different companies. For instance, an application from the energy sector is the maximization of electric vehicle adoption through the location of charging stations. Following the recent literature, the customers' total level of attraction increases when more electric vehicle charging stations are placed, as it is assumed to be highly dependent on the placement of charging infrastructure \citep{coffman2017}. Moreover, customers with an electric vehicle no longer patronize petrol stations, hence they are totally covered by the electric vehicle charging stations. This problem is of great relevance nowadays, since the European Commission has agreed on an ambitious new law to deploy sufficient alternative fuels infrastructure \citep{noauthor_new_2023} and to ban the sale of new combustion-engine cars in the bloc by 2035 \citep{noauthor_eu_2023}, as part of the European Green Deal towards zero emissions. In this work, we test the performance of our models using a case study on the adoption of electric vehicles through the installation of charging stations in the city of Trois-Rivières, Québec (Canada) \citep{anjos2022arxiv}.

Many other applications arise from companies or public services that offer a subscription/membership for a fee in exchange for the use of any available facility over a period of time. This is the case of bicycle services or the public transport network offered in cities to encourage the use of more sustainable systems, and the business model of fitness chains where a gym membership allows the use of any fitness center of the company, and the memberships make impractical patronizing facilities from different companies. It is also suitable for backup models for emergency response like the aforementioned, and for the modeling of congested systems \citep{berman2002,baron2009}. 


As stated, there is uncertainty in the partial attractions of the customers towards the facilities. This implies that when the location decision is taken by the decision maker, the covering is uncertain, since it depends on the particular attractions of the customers to each facility and the threshold of the covering (that represents the minimum level of attraction required to cover a customer). The probability functions of choice models that govern customers' attractions rely on behavioral aspects, which more often than not lead to probability distributions that require complex (i.e., non linear, non convex) mathematical formulations, or that are frequently unknown. To overcome this issue, in a similar spirit to \cite{li2018}, we propose to replace the probability distribution of the random attraction model with its empirical estimate based on a set of random samples. This simulation-based approach is known as sample average approximation (SAA)  \citep{shapiro2003} and has been applied in discrete choice models to specify the demand directly in terms of the covering functions \citep{pachecopaneque2021}. As an advantage, since we do not make any assumptions on the random attraction model, this approach allows to work with observations that are available to the decision maker even when the distribution is unknown. The basic idea of the SAA approach is to generate a sample of customers' partial attractions to then approximate the unknown attraction by the corresponding sample average. Thus, for each scenario the error terms are generated in advance and introduced in the formulation as input for each partial attraction. The resulting model is therefore deterministic, a mixed-integer problem (MIP) where a set of scenarios are included and the customers' attractions vary in each scenario. Customers are covered in the corresponding scenario if the total level of attraction surpasses the (also unknown) threshold, and the objective function is the average coverage in all the scenarios.



Our main contributions are as follows:
\begin{itemize}
    \item We give a unified modeling framework for the family of cooperative covering problems by means of the introduction of an OMf to model the total level of attraction associated to the decision maker's company. 
    \item We formulate the model as a multiperiod stochastic problem where the total level of attraction is described by an additional embedded maximization problem. In the particular case where the weights of the OMf are sorted in non-increasing order, the OMf problem can be formulated as a linear assignment problem. In this case, the model is reformulated as a MINLP.
    \item We propose a solution method to produce an equivalent mixed-integer linear problem (MILP), which we subsequently  strengthen by way of tailored valid inequalities and preprocessing techniques. 
    \item Alternatively, we propose a decomposition method based on the addition of Generalized Benders Cuts to a relaxed version of the problem with a very reduced number of variables and constraints. 
    \item We run extensive computational experiments designed to test the performance of the proposed formulations and solution techniques. Furthermore, we solve medium-size and large-scale instances of practical relevance: those from a case study on placing charging stations for electric vehicles in the city of Trois-Rivières, Québec (Canada) proposed by \citet{anjos2022arxiv}. Using this case study, we illustrate how the choice of the vector of weights in the OMf is a key decision and has a significant impact on the location of the charging stations.
\end{itemize}

The paper is organized as follows. Section \ref{sec:formulation} introduces the notation and the formulation presented. In Section \ref{sec:reformBilevel}, we derive a single-level MINLP model with the same set of optimal solutions (in terms of the location variables and the customer's decisions) than the initial model. Section \ref{sec:MILP} is devoted to the first solution approach, which includes: 1) a linearization of the MINLP model that results in a MILP model that can be solved using readily available off-the-shelf solvers (Gurobi, CPLEX, XPRESS-MP, etc.); and 2) valid inequalities and preprocessing techniques to obtain a tight and compact model. Section \ref{sec:Benders} is devoted to the second solution approach presented, a Benders' like decomposition scheme. Section \ref{sec:EstudioComputacional} comprises two computational studies, one with randomly generated small and medium-size instances designed to test and compare the solution approaches presented, and the case study on the placement of electric vehicle charging stations proposed by \citet{anjos2022arxiv}. Finally, some conclusions are stated in Section \ref{sec:conclusiones}.

\section{Formulation} \label{sec:formulation}
In this section, we formally define the mathematical programming model studied in this paper and introduce the notation. Recall that the objective is the maximization of covered demand in a cooperative setting, when a customer is covered if the weighted sum of partial attractions to the existing facilities exceeds a threshold and this weighted sum follows an ordered median function.  

Consider $\J = \{1,\dots,|\J|\}$ as the set of candidate locations of facilities. Throughout the paper, and abusing notation, $j$ is used to represent both the location $j$ and the (potential) facility located in $j$. Let $\K_j=\{1,\dots,|\K_j|\}$ represent the facility types that can be installed in $j\in \J$.  W.l.o.g., the types are ordered from the least expensive to the most expensive one. For each time period $t\in \T =\{1,\dots, |\T|\}$, assume there is a budget $b^t$ to spend on locating and/or extending facilities. Installing a facility $j$ of type $k$ at time period $t$ has an associated cost $c^t_{jk}$, with $c^t_{jk}$ non-decreasing in $k$. We assume that the facilities can only be upgraded (with a cost equal to the difference of the costs of the types in the corresponding time periods), but they cannot be eliminated or downsized.

We define binary variable $x_{jk}^t$, $\forall j\in \J$, $k\in \K_j$, $t\in \T$, equal to 1 if and only if a facility of type $k$ is installed in $j$ at time period $t$. Using these variables, the constraints associated to the location of facilities are stated as:
\begin{subequations}
\begin{align}
    & \dsum_{t' \in \T:\atop t'\le t} \dsum_{j\in \J} \dsum_{k \in \K_j} c_{jk}^{t'}(x_{jk}^{t'} - x_{jk}^{t'-1}) \leq \dsum_{t' \in \T:\atop t'\le t} b^{t'}, \quad \forall t \in \T, \label{BL1_budgettemporal}\\
    & \dsum_{k\in \K_j} x_{jk}^{t} \le  1, \quad \forall j \in \J, t \in \T, \label{BL1_outletssiexistej} \\
    &  \dsum_{k\in \K_j} kx_{jk}^{t-1} \leq \dsum_{k\in \K_j} kx_{jk}^{t} , \quad \forall j \in \J, t \in \T\setminus \{1\}, \label{BL1_xcrecienteagregadas}\\
    & x_{jk}^{t} \in \{0,1\}, \quad \forall j \in \J, t \in \T, k\in \K_j. \label{BL1_xbinaria}
\end{align}
\end{subequations}

The set of constraints \eqref{BL1_budgettemporal} guarantees that the cost of the facilities installed up to  time period $t$ does not surpass the total budget $\sum_{t'\le t} b^{t'}$ (with $x^t_{jk} = 0$ for $t=0$). By summing up on the time periods in \eqref{BL1_budgettemporal}, we allow for the surplus budget from time period $t$ to be used in subsequent time periods. Constraints \eqref{BL1_budgettemporal} are a generalization of the usual constraints fixing a number $p$ of facilities to be installed that is necessary when facilities with different installation costs that may also depend on time are considered. Constraints \eqref{BL1_outletssiexistej} ensure that only one facility can be placed in each location.  Constraints \eqref{BL1_xcrecienteagregadas} ensure that the facility of type $k$ can only be upgraded (or remain untouched) for subsequent time periods. 

Note that we have included only the simplest constraints on the location of facilities. Nevertheless, additional constraints may be required by the firm. For instance, constraint \eqref{BL1_budgettemporal} can be replaced by a constraint limiting the number of facilities of each type to open, or the company may choose to add preference constraints of the type $\sum_{k\in \K_j} x^t_{jk} \le \sum_{k\in \K_{j'}} x^t_{j'k}$ for $j,j'\in \J$, $t \in \T$, if for some reason location $j'$ is to be chosen before location $j$.

Furthermore, consider a set of classes of customers  $\I=\{1,\dots,|\I|\}$ with a homogeneous behavior, where $n_i^t$ represents the weight of class $i\in \I$ (associated, for instance, to the population of such a class) in period $t \in \T$. Since the attraction for each facility is unknown, we follow a sample average approximation method, widely used in Stochastic Programming, to estimate them. Thus, we consider a set $\S$ of scenarios with equal probabilities. As stated, a cooperative covering with an embedded Ordered Median function (OMf) is considered, and customers are covered if the resulting total attraction exceeds a threshold. Hence, for any user class $i\in \I$, time period $t\in \T$ and scenario $s\in \S$, we consider two different alternatives: $T_i^{ts}$ is a parameter representing the threshold, and $U_{i}^{ts}$ is a continuous variable that represents the total attraction associated to $i$, $t$ and $s$ as a function of the partial attractions given by each open facility. Then, defining a binary variable $z_i^{ts}$ $\forall i\in \I$, $t\in \T$, $s\in \S$, equal to 1 if and only if user class $i$ is covered in time period $t$ and scenario $s$, the constraints associated to the covering of each customer are as follows:
\begin{subequations} 
\begin{align}
     T_i^{ts} z_i^{ts} \le U_{i}^{ts} z_i^{ts}, \quad \forall i\in \I, t \in \T, s\in \S, \label{BL1_OF2sola} \\
     z_i^{ts} \in \{0,1\}, \quad \forall i\in \I, t \in \T, s\in \S, \label{BL1_zinteger}
\end{align}
\end{subequations}
\noindent and the objective can be stated as:
\begin{equation}
\max \quad \dsum_{t\in \T} \dsum_{i \in \I} n_i^t \frac{1}{|\S|} \dsum_{s \in \S} z_i^{ts}. 
\end{equation}
\noindent We remark that the multiplication of the variable $z^{ts}_i$ in the right-hand side of \eqref{BL1_OF2sola} is not needed from a modeling point of view. However, its inclusion makes fractional values of this variable infeasible when $T^{ts}_i > U^{ts}_i$, and thus gives rise to tighter relaxation bounds. Take, for instance, $T^{ts}_i=3$ and $U^{ts}_i=2$ for some $i,t,s$. Then $T^{ts}_i z^{ts}_i \le U^{ts}_i$ is satisfied by any $z^{ts}_i \in [0,2/3]$, whereas \eqref{BL1_OF2sola} is only satisfied for $z^{ts}_i=0$, so naturally the bound given by the relaxation of the problem is tighter with the non linear constraint. Furthermore, in Sections \ref{sec:MILP} and \ref{sec:Benders} we present linear reformulations of this initial model that maintain the tight bound given by \eqref{BL1_OF2sola}, and none of them requires additional variables. That is why we decided to formulate the covering of each customer through the non linear constraints \eqref{BL1_OF2sola}.

\subsubsection*{Cooperative covering}
The attraction of a customer towards the company depends on the location of the facilities, and therefore varies with the number and type of facilities placed. In order to define it, we consider the total level of attraction $U_{i}^{ts}$ as a function of the \textit{partial attractions} $u_{ij}^{ts}$ associated to each potential location $j\in \J$ of a facility. There are $k$ types of facilities that can be placed in $j$, and w.l.o.g.\ the types are ordered from the least attractive to the most attractive one. Then, each partial attraction $u_{ij}^{ts}$ is a continuous variable with a strictly positive value if and only if there exists a facility $j \in \J$ that is open for some $k \in \mathcal{K}_j$, i.e.,:
\begin{equation} \label{def_varuparciales}
    u_{ij}^{ts} := \begin{cases}  a_{ijk}^{ts}, & \text{if a facility $j\in \J$ of type $k \in \K_j$ is open}, \\
	0, & \text{otherwise,}\end{cases} \quad \forall  i \in \I, j\in \J, t\in \T, s\in \S.
 \end{equation}

Here, $a_{ijk}^{ts}$ is a parameter that estimates the partial attraction associated to placing a facility $j$ of type $k$ for the user class $i$ in time period $t$ and scenario $s$. To account for uncertainty, this attraction is divided in two parts: a measurable and deterministic part and a random non-observable one (that can be viewed as an error), i.e., $a = \hat{a} + \epsilon$. The partial attraction is only related to the distance, so it is  usually given by a general decay function $\phi(d)$ that monotonically decreases with the distance $d$. In the discrete setting, however, $\hat{a}$ can be viewed as the general attraction of the facility and can depend on many factors other than the distance, such as its size, if it is close to other establishments, etc. Finally, as stated, we assume that $a_{ijk}^{ts}$ is non-decreasing in~$k$.

Making use of the fact that only one facility can be placed in $j$, we can define the value of variable $u^{ts}_{ij}$ in terms of the location variables by means of the following equality constraint:
\begin{equation} \label{BL1_covering}
 u_{ij}^{ts} = \sum_{k\in \K_j} a_{ijk}^{ts}x_{jk}^t, \quad \forall i \in \I, j\in \J, t\in \T, s\in \S.
\end{equation}

\subsubsection*{The OMf}
As stated in the introduction, we make use of the OMf to model the total level of attraction $U_{i}^{ts}$ of customers for each $i$, $t$ and $s$. This function is a weighted sum of ordered elements, i.e.,  a mapping $\Phi_\lambda: \R^{|\J|} \rightarrow \R$ with associated weighting vector $\boldsymbol{\lambda_i} = (\lambda_{i1},\dots,\lambda_{i|\J|})$. Hence, the total level of attraction is defined as:
\begin{equation} \label{BL1_Totalcovering}
    U_{i}^{ts} := \Phi_{\boldsymbol{\lambda_i}}(u_{i1}^{ts},\dots,u_{i|\J|}^{ts}) = \sum_{j\in \J} \lambda_{ij} u_{i(j)}^{ts},
\end{equation}
\noindent where $u_{i(r)}^{ts}$ is the $r$-th largest input vector component of $u^{ts}_i$, i.e., $u_{i(1)}^{ts} \geq \ldots \geq u_{i(|\J|)}^{ts}$.

The value of vector $\boldsymbol{\lambda_i}$ is directly related to the application, and also to the assumptions made on the customer's choice rule. It can be set according to the characteristics assumed for each customer class. For instance, if we consider the vector $\boldsymbol{\lambda_i}=(1,0,\dots,0)$, then the total level of attraction $U_{i}^{ts}$ takes the value of the highest partial attraction. Then, the problem results in the classical MCLP. For more general vectors such as $\boldsymbol{\lambda_i}=(\underbrace{1,\dots,1}_{\ell},0,\dots,0)$, the total level of attraction is given by the sum of the partial attractions of the $\ell$\textit{-th closest or most attractive} facilities for a client. This setting has applications in signal-transmission facilities (such as cell-phone towers or light standards) when the partial attractions are given by a decay function $\phi(d)$ decreasing with the distance $d$ \citep{berman2009,berman2010} and up to $\ell$ facilities can cooperate in the covering. A second application arises in \cite{lin2021} in the context of the Maximum Capture Facility Location, where customers rank the facilities by non-decreasing attraction and then form a \textit{consideration set} with the $\ell$ facilities of higher rank. Finally, vectors such as $\boldsymbol{\lambda_i}=(1,\frac{1}{2},\frac{1}{4},0,\dots,0)$ correspond to customers whose total level of attraction is \textit{mainly} given by their favourite facility, but additional interesting facilities can increase the attraction to the company. It is also a generalization of backup covering models \citep{hogan1986}, where the partial attraction given by second and third options (i.e.\ backup facilities) is taken into account. We remark that OCMCLP is NP-hard because it reduces to MCLP when $\lambda_i=(1,0,\dots,0)$ $\forall i\in \I$. The NP-hardness of MCLP is proved in \citet{hochbaum1997}.

In our setting, and given that the partial attractions of a customer are ordered in non-increasing order, we can obtain a reformulation of the OMf by considering any vector $\boldsymbol{\lambda_i}$ \citep[see][]{fernandez2013}. For this, define the binary variables $\sigma_{ijr}^{ts}$ $\forall i\in\I$, $t\in\T$, $s\in \S$, $j,r\in \J$. Then $\sigma_{ijr}^{ts}=1$ if and only if $u^{ts}_{ij}$ is the $r$-th largest partial attraction for customer class $i\in \I$. With these variables, the integer model is:
\begin{subequations} \label{OMf}
\begin{align}
     U_{i}^{ts} = \quad\max_{\boldsymbol{\sigma}_i^{ts}} \quad & \dsum_{j \in \J} \dsum_{r \in \J}\lambda_{ir} u_{ij}^{ts}  \sigma_{ijr}^{ts} \label{OMf_OF3}\\
     \text{s.t.} \quad & \dsum_{j \in \J} \sigma_{ijr}^{ts} = 1,\quad \forall r \in \J, \label{OMf_assignmentj}\\
    &\dsum_{r\in \J} \sigma_{ijr}^{ts} = 1,\quad \forall j \in \J, \label{OMf_assignmentr}\\	
    & \dsum_{j \in \J} u_{ij}^{ts}  \sigma_{ijr-1}^{ts} \ge \dsum_{j \in \J} u_{ij}^{ts} \sigma_{ijr}^{ts}, \quad \forall r\in \J \setminus \{1\}, \label{OMf_orderpartial coverings}\\
    & \sigma_{ijr}^{ts} \in \{0,1\},\quad \forall j,r \in \J. \label{OMf_sigmabinary} 
\end{align}
\end{subequations}

However, we consider only non-increasing vectors $\boldsymbol{\lambda_i}$, i.e., any $\boldsymbol{\lambda_i} \ge \boldsymbol{0}$ such that $\lambda_{i1} \ge\dots \ge \lambda_{i|\J|}$. The reason is that, realistically, the partial attraction of a customer towards a specific facility (which can be seen e.g.\ as the percentage of times they make use of said facility) decreases when bigger/closer facilities are installed. Hence, if customers obtain their total level of attraction by summing up the weighted partial attractions, they will likely penalize facilities with a lower partial attraction (such as the smallest/farthest facilities). Besides, defining the OMf as en embedded optimization problem with a fixed monotone $\boldsymbol{\lambda_i}$ and parameters $a_{ijk}^{ts}$ non-decreasing in $k$ guarantees that the total level of attraction is non-decreasing when more facilities are located throughout time. This is also a realistic assumption that guarantees some consistency in the model. Finally, when the entries of the vector $\boldsymbol{\lambda_i}$ are non-increasing and the ordering of the partial attractions with respect to $j$ as well, $\Phi_{\boldsymbol{\lambda_i}}$  can be stated as an assignment problem, i.e., problem \eqref{OMf} without constraints \eqref{OMf_orderpartial coverings}.

The complete model (OCMCLP) proposed is then: 
\begin{subequations} \label{BL1}
\begin{align}
\text{(OCMCLP)} \quad \max_{\boldsymbol{x},\boldsymbol{u}, \boldsymbol{z}}\quad & \dsum_{t\in \T} \dsum_{i \in \I} n_i^t \frac{1}{|\S|} \dsum_{s \in \S} z_i^{ts} \label{BL1_OF1}\\
    \text{s.t.}\quad & \eqref{BL1_budgettemporal}-\eqref{BL1_xbinaria}, \eqref{BL1_OF2sola}-\eqref{BL1_zinteger}, \eqref{BL1_covering}, \label{BL1_constraintsMaster}\\
    & U_{i}^{ts} = \quad\max_{\boldsymbol{\sigma}_i^{ts}} \quad \dsum_{j \in \J} \dsum_{r \in \J}\lambda_{ir} u_{ij}^{ts}  \sigma_{ijr}^{ts} \label{BL1_OF3}\\
    &\hspace{18mm}  \text{s.t.} \quad  \dsum_{j \in \J} \sigma_{ijr}^{ts} = 1,\quad \forall r \in \J, \label{BL1_assignmentj}\\
    &\hspace{18mm}  \phantom{s.t.}\quad\dsum_{r\in \J} \sigma_{ijr}^{ts} = 1,\quad \forall j \in \J, \label{BL1_assignmentr}\\
	& \hspace{18mm}  \phantom{s.t.} \quad \sigma_{ijr}^{ts} \in \{0,1\},\quad \forall j,r \in \J. \label{BL1_sigmabinary}
\end{align}
\end{subequations}


Model (OCMCLP) has a nested optimization problem \eqref{BL1_OF3}-\eqref{BL1_sigmabinary} designed to obtain the value of $U^{ts}_i$. The values of the partial attractions $u^{ts}_{ij}$ are uniquely determined by variables $x^t_{jk}$ and act as \textit{parameters} in the objective function \eqref{BL1_OF3}. We reformulate it in the next section, obtaining a single-level mixed-integer linear formulation that can be solved using modern general-purpose MILP solvers. 

To emphasize the relevance of adequately choosing the vector $\boldsymbol{\lambda_i}$, we have included Example~\ref{ex:ejemplo1}, which shows an optimal solution of the same instance with different vectors $\boldsymbol{\lambda_i}$.

\begin{example} \label{ex:ejemplo1}
 In the toy instance considered, $|\T|=|\S|=1$, so we remove the indices $t$, $s$ from the variables and parameters. Furthermore, $|\J|=2$ and $|\K_j|=3$ for all $j$, and the costs $c_{jk}$ of opening any facility $j$ of type $k=1,2,3$ are, respectively, 2,3,5, $\forall j$. The budget is $b=5$, so in any feasible solution we can place up to one facility of type 3, or up to two facilities of types $k=1,2$.  

As for the customer classes, $|\I|=3$, $n_i=1$ $\forall i\in \I$ (so we identify customer classes with customers) and the threshold $T_i$ is the same for all the customers, $T_i=3$ $\forall i$. The partial attractions of all the customers for each facility and type can be seen in Table \ref{tab:ejemplo1partial coverings}. For instance, the partial attraction $a_{323}=3.5$. 

\begin{table}[H]
\begin{center}
\begin{tabular}{ccrrrrrr} 
\toprule
Customers & $T_i$ & \multicolumn{3}{c}{$j=1$} & \multicolumn{3}{c}{$j=2$}\\ 
\cmidrule(lr){3-5} \cmidrule(lr){6-8}
      &   & $k=1$ & $k=2$ & $k=3$ & $k=1$ & $k=2$ & $k=3$  \\ \hline
$i=1$ & 3 &     2 &   2.5 &     3 &     1 & 1.5   & 2 \\ 
$i=2$ & 3 &     2 &     3 &     4 &     1 & 1.5   & 2 \\ 
$i=3$ & 3 &   1.5 &     2 &   2.5 &   2.5 & 3     & 3.5 \\
\bottomrule
\end{tabular}
\caption{Attraction matrix $(a_{ijk})$ and vector of thresholds $T_i$ for Example \ref{ex:ejemplo1}.} 
\label{tab:ejemplo1partial coverings}
\end{center}
\end{table}

We have solved this instance for two different $\boldsymbol{\lambda_i}$ vectors and show the optimal placement of facilities in Figure \ref{fig:example1}. Note that, for ease of illustration, customer 1 is represented as $i_1$ and facility $1$ is represented as $j_1$ (and so on). In Figure \ref{fig:ex1a}, the total level of attraction for each customer only depends on the partial attraction of their most relevant station, i.e., $\boldsymbol{\lambda_i} = (1,0)$ $\forall i\in \I$. This is the classical non-cooperative MCLP. In this case, the optimal solution consists of placing one facility of type $k=3$ in $j=1$, and customers 1 and 2 are covered (i.e., $z_1=z_2=1$, $z_3=0$).  The optimal value (the number of customers covered in this example) is equal to 2. 

\begin{figure}
    \centering
    \begin{subfigure}[b]{.48\linewidth}
    \centering
    \fbox{\includegraphics[scale = 5.5]{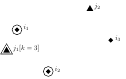}}
    \caption{$\lambda_i = (1,0)$.}\label{fig:ex1a}
    \end{subfigure}~
    \begin{subfigure}[b]{.48\linewidth}
    \centering 
    \fbox{\includegraphics[scale = 5.5]{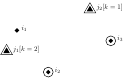}}
    \caption{$\lambda_i = (0.9,0.5)$.}\label{fig:ex1b}
    \end{subfigure}
    \caption{Illustrative example of solutions obtained using different $\boldsymbol{\lambda_i}$.}
    \label{fig:example1}
\end{figure}

In Figure \ref{fig:ex1b}, we assume that the total level of attraction is given by an aggregation of the partial attractions of the two most relevant facilities for each customer, weighted using $\boldsymbol{\lambda_i} = (0.9,0.5)$ $\forall i\in \I$. In this setting, the optimal placement of facilities is given by opening a facility of type $k=2$ in $j=1$ and a facility of type $k=1$ in $j=2$. In this case, customer 1 prefers facility 1 to 2, so the total level of attraction is $U_1= \lambda_{11} u_{1(1)} + \lambda_{12}u_{1(2)}  = \lambda_{11}a_{112} +  \lambda_{12}a_{121} = 0.9 \cdot 2.5 + 0.5 \cdot 1= 2.75$, so customer 1 is not covered in this solution. However, customer 3 prefers facility 2 over facility 1, thus $U_3 = \lambda_{31} u_{3(1)} + \lambda_{32}u_{3(2)} =0.9a_{321} + 0.5 a_{312} = 0.9 \cdot 2.5 + 0.5 \cdot 2 = 3.25$, so customer 3 is covered (and so is customer 2). The optimal value is also 2.

As illustrated, $\boldsymbol{\lambda_i}=(1,0)$ favors the location of fewer but bigger/more attractive facilities, whereas other vectors tend to favor solutions with more facilities of smaller size. The decision maker can choose the adequate $\boldsymbol{\lambda_i}$ depending on the setting, the type of customer classes (if they base their covering on a single facility or on a combination of several of them) and the desired solutions.

\end{example}

\section{Reformulation of model (OCMCLP) into a MINLP} \label{sec:reformBilevel}
In formulation \eqref{BL1}, the value of the total level of attraction variable $U^{ts}_i$ is obtained as the solution of an assignment problem. To obtain a single-level formulation, we consider fixed $i$, $t$ and $s$ and we focus on obtaining the value of $U^{ts}_i$ for fixed partial attraction values $u_{ij}^{ts}$. 
\begin{prop} \label{prop:binivelaunnivel}
Consider the following single-level MINLP:
\begin{subequations} \label{SingleLevelNoLineal}
\begin{align} 
\max_{\boldsymbol{x},\boldsymbol{z},\boldsymbol{u},\boldsymbol{\sigma}}\quad & \dsum_{t\in \T} \dsum_{i \in \I} n_i^t \frac{1}{|\S|} \dsum_{s \in \S} z_i^{ts} \label{SingleLevelNoLineal_OF}\\
\text{\rm s.t.}\quad & 
\eqref{BL1_budgettemporal}-\eqref{BL1_xbinaria}, \eqref{BL1_covering}, \label{SingleLevelNoLineal_MainConstraints}\\
\quad & T_i^{ts}z_i^{ts} \le U_{i}^{ts}z_i^{ts}, \quad \forall i \in \I, t\in \T, s\in \S,  \label{SingleLevelNoLineal_rest_z0}\\
& U_{i}^{ts} = \dsum_{j \in \J} \dsum_{r \in \J}\lambda_{ir} u_{ij}^{ts}  \sigma_{ijr}^{ts},\quad \forall i \in \I, t\in \T, s\in \S,   \label{SingleLevelNoLineal_U}  \\
& \dsum_{j \in \J} \sigma_{ijr}^{ts} \le 1,\quad \forall i \in \I, t\in \T, s\in \S, r \in \J,  \label{SingleLevelNoLineal_sumasigmaj2}\\
&\dsum_{r\in \J} \sigma_{ijr}^{ts} \le 1,\quad \forall i \in \I, t\in \T, s\in \S, j \in \J, \label{SingleLevelNoLineal_sumasigmar2} \\
& \sigma_{ijr}^{ts} \in [0,1], \quad \forall i \in \I, t\in \T, s\in \S, j,r \in \J, \label{SingleLevelNoLineal_sigmacontinuas} \\
& z_i^{ts} \in [0,1], \quad \forall i \in \I, t\in \T, s\in \S. \label{SingleLevelNoLineal_zcontinuas}
\end{align}
\end{subequations}
Problem \eqref{SingleLevelNoLineal} is a relaxation of problem \eqref{BL1} with the same set of optimal solutions in terms of $(x,z)$.
\end{prop}

\begin{proof}
Consider a location vector $\bar{x}$ satisfying the location constraints \eqref{BL1_budgettemporal}-\eqref{BL1_xbinaria}, and the values of the partial attractions $\bar{u}^{ts}_{ij}$ given by \eqref{BL1_covering}. Then for fixed $i$, $t$, $s$, the values of $\bar{U}^{ts}_i$, $\bar{z}^{ts}_i$ are univocally determined by the assignment problem and constraint \eqref{BL1_OF2sola} respectively. To prove the statement, it suffices to see that (i) if $\bar{z}^{ts}_i=0$ in \eqref{BL1}, then constraints \eqref{SingleLevelNoLineal_rest_z0}-\eqref{SingleLevelNoLineal_zcontinuas} guarantee $z^{ts}_i=0$ in \eqref{SingleLevelNoLineal}, and (ii) if $\bar{z}^{ts}_i=1$ in \eqref{BL1}, then there exists a feasible solution of \eqref{SingleLevelNoLineal} with $z^{ts}_i=1$.

Let us first reformulate the assignment problem \eqref{BL1_OF3}-\eqref{BL1_sigmabinary}. To begin with, given that the constraint matrix is totally unimodular and the right-hand side vector is integer, we can relax the integrality constraints on the assignment variables $\sigma$ for fixed values of $u_{ij}^{ts}$, obtaining a linear assignment problem. Notice also that the nonnegativity assumption on the partial attractions $u^{ts}_{ij}$ and the vector $\boldsymbol{\lambda_i}$ implies that the linear equality constraints of the assignment problem can be relaxed to be less than or equal to 1.

Second, the aim of this auxiliary problem is to provide  $\bar{U}^{ts}_i := \max_{\sigma} \sum_{j \in \J} \sum_{r \in \J}\lambda_{ir} \bar{u}_{ij}^{ts}  \sigma_{ijr}^{ts}$, which in turn takes part in constraints \eqref{BL1_OF2sola} and is used to derive the value of $z^{ts}_i$. But the value of $z_i^{ts}$ only depends on the difference $\bar{U}_{i}^{ts}-T_i^{ts}$, and not on the actual value of $\bar{U}_{i}^{ts}$: if $\bar{U}_i^{ts} - T_i^{ts}\ge 0$, then $z_i^{ts}=1$. Therefore, for fixed partial attractions $\bar{u}_{ij}^{ts}$ two possibilities can occur:
\begin{itemize}
    \item[i)] There exists an assignment $\bar{\sigma}$ such that $\sum_{j \in \J} \sum_{r \in \J}\lambda_{ir} \bar{u}_{ij}^{ts}  \bar{\sigma}_{ijr}^{ts} \ge T_i^{ts}$. In this case, $\bar{z}_i^{ts}=1$ in \eqref{BL1_OF2sola} because $\bar{U}^{ts}_i  \ge$ $\sum_{j \in \J} \sum_{r \in \J}\lambda_{ir} \bar{u}_{ij}^{ts}  \bar{\sigma}_{ijr}^{ts} \ge T_i^{ts}$.
    \item[ii)] For any assignment $\sigma$, it holds $\sum_{j \in \J} \sum_{r \in \J}\lambda_{ir} \bar{u}_{ij}^{ts}  \sigma_{ijr}^{ts} < T_i^{ts}$. In this case, $\bar{U}^{ts}_i < T_i^{ts}$ and constraint \eqref{BL1_OF2sola} implies $z_i^{ts}=0$ regardless of the assignment chosen.       
\end{itemize}
Thus, we can consider the relaxation of the assignment problem given by \eqref{SingleLevelNoLineal_U}-\eqref{SingleLevelNoLineal_sigmacontinuas}. If $i)$\ holds, then the maximization of $z$ guarantees that an assignment $\bar{\sigma}$ such that $U^{ts}_i = \sum_{j \in \J} \sum_{r \in \J}\lambda_{ir} u_{ij}^{ts}  \bar{\sigma}_{ijr}^{ts} \ge T_i^{ts}$ is chosen in any optimal solution. If $ii)$\ holds, any assignment satisfying \eqref{SingleLevelNoLineal_sumasigmaj2}-\eqref{SingleLevelNoLineal_sigmacontinuas} gives rise to a feasible solution of \eqref{SingleLevelNoLineal} with $z^{ts}_i=0$, and $U^{ts}_i$ can take values in the interval $[0,\bar{U}^{ts}_i]$. Note that formulation \eqref{SingleLevelNoLineal} has feasible solutions with assignments which are infeasible in \eqref{BL1}.

Finally, note that the integrality constraint on $z^{ts}_i$ can be relaxed in \eqref{BL1}. Indeed, when $T_i^{ts} > U^{ts}_i$, constraint \eqref{BL1_OF2sola} forces $z^{ts}_i=0$, and when $T_i^{ts} \le U^{ts}_i$, the maximization of the objective function will lead to $z^{ts}_i=1$. We have chosen to include the non-linear constraint \eqref{BL1_OF2sola} instead of its linear version 
\begin{equation*}
T_i^{ts}z_i^{ts} \le U_{i}^{ts}, \quad \forall i \in \I, t\in \T, s\in \S
\end{equation*}
\noindent to obtain a tighter model where the integrality constraints on these variables can be relaxed.

\end{proof}

\section{First solution approach: a mixed-integer linear formulation} \label{sec:MILP}
In this section, we propose a linearization of formulation \eqref{SingleLevelNoLineal} that results in a MILP which can be solved using off-the-shelf solvers. We also introduce several families of valid inequalities and some preprocessing techniques to obtain a tight and compact formulation. A comparison between the MILP model and the model with the valid inequalities is included in Section \ref{sec:EstudioComputacional}.

\subsection*{Linearization of the single-level MINLP \eqref{SingleLevelNoLineal} }
Problem \eqref{SingleLevelNoLineal} is non-linear due to constraints \eqref{SingleLevelNoLineal_rest_z0} and \eqref{SingleLevelNoLineal_U}. We carry out a linearization resulting in the MILP model stated in the following proposition.
\begin{prop} \label{prop_linearizacion}
Consider the following MILP:
\begin{subequations} \label{SL}
\begin{align}
    \max\quad & \dsum_{t\in \T} \dsum_{i \in \I} n_i^t \frac{1}{|\S|} \dsum_{s\in \S} z_i^{ts} \label{SL_OF}\\
    \text{s.t.}\quad & \dsum_{t' \in \T:\atop t'\le t} \dsum_{j\in \J} \dsum_{k \in \K_j} c_{jk}^{t'}(x_{jk}^{t'} - x_{jk}^{t'-1}) \leq \dsum_{t' \in \T:\atop t'\le t} b^{t'}, \quad \forall t \in \T,  \label{SL_budgettemporal}\\
    & \dsum_{k\in \K_j} x_{jk}^t \le 1, \quad \forall j \in \J, t \in \T, \label{SL_outletssiexistej} \\
    & \dsum_{k\in \K_j} kx_{jk}^{t-1} \le \dsum_{k\in \K_j} kx_{jk}^t, \quad \forall j \in \J, t \in \T\setminus\{1\}, \label{SL_xcrecienteagregadas}\\
    &u_{ij}^{ts} = \sum_{k\in \K_j} a_{ijk}^{ts}x_{jk}^t, \quad \forall i \in \I, j\in \J, t\in \T, s\in \S, \label{SL_coveringUpperBoundX}\\
     &  T_i^{ts} z^{ts}_i \le \sum_{j \in \J}\sum_{r \in \J} \lambda_{ir} w_{ijr}^{ts} , \quad \forall i \in \I, t\in \T, s\in \S, \label{SL_zZeroLineal} \\
     & \dsum_{j \in \J} \sigma_{ijr}^{ts} \le z^{ts}_i,\quad \forall i \in \I, t\in \T, s\in \S, r \in \J,  \label{SL_sumasigmaj}\\
    &\dsum_{r\in \J} \sigma_{ijr}^{ts} \le z^{ts}_i,\quad \forall i \in \I, t\in \T, s\in \S, j \in \J, \label{SL_sumasigmar} \\
    &  w_{ijr}^{ts} \le a_{ij|\K_j|}^{ts}\sigma_{ijr}^{ts}, \quad \forall i \in \I, t\in \T, s\in \S, j,r \in \J, \label{SL_rest_ulinear_cotasuperiorz2}\\
    &  w_{ijr}^{ts} \le u_{ij}^{ts}, \quad \forall i \in \I, t\in \T, s\in \S, j,r \in \J,  \label{SL_rest_ulinear_cotasuperiorw2}\\
     & x^t_{jk} \in \{0,1\}, \quad \forall j\in \J, t\in \T, k\in \K_j, \label{SL_xybinarias}\\
    & z_i^{ts} \in [0,1], \quad \forall i \in \I, t\in \T, s\in \S, \label{SL_zcontinua} \\
    &  w_{ijr}^{ts} \ge 0, \quad \forall i \in \I,  t\in \T, s\in \S, j,r\in \J,\label{SL_coveringLowerBound0}\\
    & \sigma_{ijr}^{ts} \in \{0,1\}, \quad \forall i \in \I, t\in \T, s\in \S, j,r\in \J. \label{SL_sigmabinaria}
\end{align}
\end{subequations}
Model \eqref{SL} is a linearization of \eqref{SingleLevelNoLineal} where a new set of auxiliary variables $w^{ts}_{ijk}:= u^{ts}_{ij}\sigma^{ts}_{ijk}$, $\forall j,r,\in \J$ has been included.  
\end{prop}
\begin{proof}
    The linearization is carried out in two steps. First, we apply a perspective transformation in the manner of \cite{gunluk2012} to linearize constraints \eqref{SingleLevelNoLineal_rest_z0}. This transformation allows us to keep the relaxation of the integrality constraint on the $z$-variables (i.e.\ constraints \eqref{SL_zcontinua}). The second step is a standard linearization of the bilinear terms in constraints \eqref{SingleLevelNoLineal_U} in the manner of \cite{mccormick1976}. This step gives rise to constraints \eqref{SL_rest_ulinear_cotasuperiorz2}, \eqref{SL_rest_ulinear_cotasuperiorw2} and \eqref{SL_coveringLowerBound0}, and can only be applied if variables $\sigma$ are binary, thus we circle back to \eqref{SL_sigmabinaria}. The detailed proof can be found in \ref{app:tmaAlgoritmoBenders}.
\end{proof}

\subsection{Valid inequalities and preprocessing techniques for model \eqref{SL}.} \label{sec:VVII}
In the following, we introduce several sets of valid inequalities for model \eqref{SL}. 

\begin{prop}
For each $i \in \I$, $t\in \T$, $s\in \S$, the set of inequalities
\begin{equation} \label{ddvv_cotautilidad}
   \sum_{r\in \J} w_{ijr}^{ts} \le u^{ts}_{ij}, \quad \forall  j \in \J,
\end{equation}
\noindent is valid for formulation \eqref{SL}. Furthermore, they dominate constraints \eqref{SL_rest_ulinear_cotasuperiorw2}.
\end{prop}
\begin{proof}
    The proof of the validity is straightforward considering the definition of $w$ and constraints \eqref{SL_sumasigmar}:
    \begin{equation*}
        \sum_{r\in \J} w_{ijr}^{ts} := \sum_{r\in \J} u^{ts}_{ij}\sigma^{ts}_{ijr} = u^{ts}_{ij} \sum_{r\in \J}\sigma^{ts}_{ijr} \le u^{ts}_{ij}.
    \end{equation*}
    \noindent Likewise, they dominate \eqref{SL_rest_ulinear_cotasuperiorw2} because the right-hand side of the constraints is the same, but the left-hand side of \eqref{ddvv_cotautilidad} has a sum of the non-negative variables $w_{ijr}^{ts}$.
\end{proof}

\begin{prop}
For each $i \in \I$, $t\in \T$, $s\in \S$, the set of inequalities
\begin{equation} \label{ddvv_outlets}
   w_{ijr}^{ts} \le a_{ijk}^{ts}\sigma_{ijr}^{ts} + \sum_{k' \in \K_j:\atop k' > k} (a_{ijk'}^{ts} - a_{ijk}^{ts})x^{t}_{jk'}, \quad \forall  j, r \in \J, k\in \K_j,
\end{equation}
\noindent is valid for formulation \eqref{SL}. 
\end{prop}
\begin{proof}
For given $i \in \I$, $t\in \T$, $s\in \S$, $j,r\in \J$, let us prove that the right-hand side of \eqref{ddvv_outlets} is an upper bound on the value of $w_{ijr}^{ts}$ $\forall k \in \K_j$. Two situations can arise, depending on the values of $\sigma_{ijr}^{ts}$ and $x^t_{jk}$:
\begin{itemize}
    \item If $\sigma_{ijr}^{ts}=0$ or $\sum_{k\in \K_j} x^t_{jk}=0$, then $w_{ijr}^{ts}=0$ and the right-hand side of the constraint is non-negative (because $a$ is non-decreasing in $k$ by assumption), so the constraint holds. 
    \item Otherwise, $w_{ijr}^{ts}\le a_{ij\bar{k}}^{ts}$ for a given $\bar{k}\in \K_j$ such that $x^{t}_{j\bar{k}}=1$. Then for $1 \le k<\bar{k}$, the right-hand side of \eqref{ddvv_outlets} is 
    $$a_{ijk}^{ts}\sigma_{ijr}^{ts} + \sum_{k' \in \K_j: k' > k} (a_{ijk'}^{ts} - a_{ijk}^{ts})x^{t}_{jk'} = a_{ijk}^{ts} + (a_{ij\bar{k}}^{ts} - a_{ijk}^{ts})x^{t}_{j\bar{k}} = a_{ij\bar{k}}^{ts}.$$
    And for $k\ge\bar{k}$, the right-hand side of \eqref{ddvv_outlets} becomes $a_{ijk}^{ts}$ with $a_{ijk}^{ts} \ge a_{ij\bar{k}}^{ts}$, so \eqref{ddvv_outlets} is valid $\forall k$.
\end{itemize}
\end{proof}
The previous valid inequalities are of special relevance because, apart from strengthening the bound of the linear relaxation of the problem, they allow us to relax the integrality constraints on the $\sigma$ variables:
\begin{prop} \label{prop:asignacionrelajada}
The integrality constraints \eqref{SL_sigmabinaria} can be relaxed in formulation \eqref{SL} if we include valid inequalities \eqref{ddvv_outlets}. 
\end{prop}
\begin{proof}
Let us prove, for fixed $i\in \I$, $t\in\T$, $s\in\S$, that the maximum value that the sum $\sum_{j \in \J} \sum_{r \in \J}\lambda_{ir} w_{ijr}^{ts}$ attains in formulation \eqref{SL} with inequalities \eqref{ddvv_outlets} is bounded by the value that the total level of attraction $U^{ts}_i$ takes in the MINLP \eqref{SingleLevelNoLineal}.

If $z^{ts}_i=0$, then \eqref{SL_sumasigmaj}-\eqref{SL_sumasigmar} imply $\boldsymbol{\sigma}=0$, so the assignment is integer and $U^{ts}_i=0$. As for $z^{ts}_i=1$, let $\bar{x}$ be a feasible (integer) solution of the first-level problem, i.e., an integer vector satisfying \eqref{SL_budgettemporal}-\eqref{SL_xcrecienteagregadas}, and fixed values $\bar{u}^{ts}_{ij}$. 

For a given $j\in \J$, if $\sum_{k\in \K_j} x^t_{jk} = 0$, then by constraints \eqref{SL_rest_ulinear_cotasuperiorw2} and \eqref{SL_coveringLowerBound0} it holds $w^{ts}_{ijr} = 0$. Otherwise, let $k_j \in \K_j$ be the unique $k$ such that $\bar{x}^t_{jk_j}=1$. Then, constraint $k_j$ from set \eqref{ddvv_outlets} is
\begin{equation*}
    w^{ts}_{ijr} \le a_{ijk_j}^{ts}\sigma_{ijr}^{ts} = \bar{u}^{ts}_{ij}\sigma_{ijr}^{ts}, 
\end{equation*}
\noindent and therefore $\sum_{j \in \J} \sum_{r \in \J}\lambda_{ir} w_{ijr}^{ts} \le \sum_{j \in \J} \sum_{r \in \J}\lambda_{ir} u^{ts}_{ij}\sigma_{ijr}^{ts} = U^{ts}_i$ in \eqref{SingleLevelNoLineal}.
\end{proof}

\begin{prop}
For each $i \in \I$, $t\in \T$, $s\in \S$, the set of inequalities
\begin{equation} \label{ddvv_outletsSumada}
   \sum_{r\in \J} w_{ijr}^{ts} \le \sum_{r\in \J}a_{ijk}^{ts}\sigma_{ijr}^{ts} + \sum_{k' \in \K_j:\atop k' > k} (a_{ijk'}^{ts} - a_{ijk}^{ts})x^{t}_{jk'}, \quad \forall  j\in \J, k\in \K_j,
\end{equation}
\noindent is valid for formulation \eqref{SL}. 
\end{prop}
\begin{proof}
For a fixed $j\in \J$, we distinguish two cases. If $\sum_{r\in \J}\sigma_{ijr}^{ts}=0$ or $\sum_{k' \in \K_j} x^{t}_{jk'}=0$, then $\sum_{r\in \J} w_{ijr}^{ts}=0$ and the right-hand side of \eqref{ddvv_outletsSumada} is non-negative, so the constraints are valid. In other case, $\sum_{r\in \J}\sigma_{ijr}^{ts}=\sum_{k \in \K_j} x^{t}_{jk}=1$, so there exist $\bar{r}$ and $\bar{k}$ such that $\sigma_{ij\bar{r}}^{ts}=x^{t}_{j\bar{k}}=1$. In this case, the right-hand side of \eqref{ddvv_outletsSumada} needs to be an upper bound of $\sum_{r\in \J} w_{ijr}^{ts} = w_{ij\bar{r}}^{ts}$ for all $k\in \K_j$. Again we distinguish two cases. For $k<\bar{k}$,   $\sum_{r\in \J}a_{ijk}^{ts}\sigma_{ijr}^{ts} + \sum_{k' \in \K_j: k' > k} (a_{ijk'}^{ts} - a_{ijk}^{ts})x^{t}_{jk'}=a_{ijk}^{ts} + (a_{ij\bar{k}}^{ts}-a_{ijk}^{ts})$. And for $k\ge\bar{k}$,  the right-hand side of \eqref{ddvv_outletsSumada} is equal to $a_{ijk}^{ts}$, an upper bound on $a_{ij\bar{k}}^{ts}$.
\end{proof}

\begin{prop}
The following family of inequalities
\begin{equation} \label{eq:SL_xcrecientedesagregadasVVII}    
\dsum_{k' \in \K_j: \atop k' \geq k} x_{jk'}^{t-1} \le \dsum_{k' \in \K_j: \atop k' \geq k} x_{jk'}^{t},\quad  \forall j \in \J, t \in \T \setminus\{1\}, k\in \K_j, 
\end{equation}
\noindent is valid for formulation \eqref{SL} and dominates constraints \eqref{SL_xcrecienteagregadas}.
\end{prop}
\begin{proof}
The validity of \eqref{eq:SL_xcrecientedesagregadasVVII} is straightforward using \eqref{SL_outletssiexistej} and \eqref{SL_xcrecienteagregadas}. To prove their dominance over \eqref{SL_xcrecienteagregadas}, it suffices to note that, for a fixed $j$ and $t$, the corresponding constraint from \eqref{SL_xcrecienteagregadas} is obtained by summing up the subset of constraints from \eqref{eq:SL_xcrecientedesagregadasVVII} for all $k\in \K_j$.
\end{proof}

As previously stated, the $\sigma$ variables are just auxiliary variables used to compute the value of $U^{ts}_i$. Therefore, we can develop inequalities that bound the values of the assignment variables as long as the maximum value that variable $U^{ts}_i$ attains for a given solution of the first-level problem remains unaltered. The following two propositions are developed with this purpose:
\begin{prop} \label{prop:eqSL_wacotadasporyVVII}
For each $i \in \I$, $t\in \T$, $s\in \S$, the set of inequalities
\begin{equation} \label{eq:SL_wacotadasporyVVII}   
\sum_{r\in \J} \sigma^{ts}_{ijr} \le \sum_{k\in \K_j} x^t_{jk}, \quad \forall j\in \J, 
\end{equation}
\noindent is valid for formulation \eqref{SL}, in the sense that it does not eliminate feasible solutions in terms of the variables $x,z$.
\end{prop}
\begin{proof}
If $\sum_{k\in \K_j} x^t_{jk}=1$, then \eqref{eq:SL_wacotadasporyVVII} is redundant. And for $\sum_{k\in \K_j} x^t_{jk}=0$, then $\sum_{r\in\J} w^{ts}_{ijr}=0$ regardless of the value of $\sum_{r\in \J} \sigma^{ts}_{ijr}$, so the later sum can be set to zero.
\end{proof}

Finally, we derive some preprocessing of the problem that allows to eliminate variables and constraints for particular cases of $\boldsymbol{\lambda_i}$.
\begin{prop} \label{prop:preproNoOrdenCompleto}
If $\lambda_{ir}=0$, then $\sigma^{ts}_{ijr}$ need not be defined in formulation \eqref{SL} $\forall i\in \I$, $t\in \T$, $s\in \S$, $j\in \J$.
\end{prop}
\begin{proof}
The result follows using a similar reasoning to that of Proposition \ref{prop:eqSL_wacotadasporyVVII}.
\end{proof}

Proposition \ref{prop:preproNoOrdenCompleto} allows us to obtain a simplified model when there is no need to order all the partial attractions of the customers because only a subset of them take part in the computation of the total level of attraction.

\section{Second solution approach: Benders Decomposition} \label{sec:Benders}
In this section, we propose a Benders-like decomposition approach to solve model \eqref{SingleLevelNoLineal}. The standard Benders recipe consists of projecting out all the continuous variables (i.e., all the variables except for vector $x$) and the associated constraints in \eqref{SingleLevelNoLineal}. However, we follow a different approach. Recall that for fixed values of the location variables $x$, our problem is decomposable by customer $i$, time period $t$ and scenario $s$. Since there is only one customer decision variable $z^{ts}_i$ per subproblem and they take part in the objective function, we leave variables $x$ and $z$ in the master problem and project out the variables and constraints associated to the assignment problem used to characterize the OMf. This translates to solving the master problem

\begin{subequations} \label{MasterProblem}
\begin{align} 
{\rm (MP)} \quad \max\quad & \dsum_{t\in \T} \dsum_{i \in \I} n_i^t \frac{1}{|\S|} \dsum_{s \in \S} z_i^{ts} \label{MasterProblem_OF}\\
\text{\rm s.t.}\quad & 
\eqref{BL1_budgettemporal}-\eqref{BL1_xbinaria}, \label{MasterProblem_MainConstraints}\\
\quad & \B^{ts}_i(x,z)\ge 0, \quad \forall i \in \I, t\in \T, s\in \S,  \label{MasterProblem_rest_z0}\\
& z_i^{ts} \in [0,1], \quad \forall i \in \I, t\in \T, s\in \S, \label{MasterProblem_zcontinuas}
\end{align}
\end{subequations}
\noindent where $\B^{ts}_i(x,z)$ represents the Benders concave function bounding variable $z^{ts}_i$ by the total level of attraction $U^{ts}_i$ given by the OMf assignment problem. For each $i$, $t$, $s$ and at each iteration, $\B^{ts}_i(\bar{x}, \bar{z})$ is a feasibility cut associated to the dual of each assignment subproblem for a fixed solution $(\bar{x},\bar{z})$ of the relaxation of the master problem \eqref{MasterProblem}: 
\begin{equation} \label{BendersSubproblem_FeasibilityForm}
 {\rm (SUB)}^{ts}_i \quad  \max \left\{ 0: \eqref{BL1_covering},  \eqref{SingleLevelNoLineal_rest_z0}-\eqref{SingleLevelNoLineal_sigmacontinuas} \right\}.
\end{equation}

Instead of solving the dual of problems \eqref{BendersSubproblem_FeasibilityForm}, we seek for a different normalization of the Benders cuts that appears naturally in our problem. For that, we follow the reasoning of the proof of Proposition \ref{prop:binivelaunnivel} and exploit the fact that a solution of the master \eqref{MasterProblem} is feasible if and only if it satisfies constraint $T_i^{ts}z^{ts}_i \le U_i^{ts}z^{ts}_i$. The latter constraint is non-linear, but it is linear (and hence concave) for fixed values of $U$. Furthermore, since $U_i^{ts}$ is the total level of attraction associated to the location of the facilities, it can be seen as a concave function of $x$: $U_i^{ts}(x)$. Therefore, we can approximate the non-linear constraints by linear (outer approximation) cuts that are generated and added on the fly to feasible (possibly non-integer) solutions of the master problem. Thus, for fixed values of the location variables $\bar{x}$, $U_i^{ts}$ can be overestimated by a supporting hyperplane at $\bar{x}$ and the following linear cut can be obtained:
\begin{multline} 
T_i^{ts}z^{ts}_i \le U_i^{ts}(x)z^{ts}_i \le \left(U_i^{ts}(\bar{x}) + \sum_{j\in \J}\sum_{k\in \K_j} \bar{s}_{jk} (x^t_{jk} - \bar{x}^t_{jk})\right)z^{ts}_i \le \\
\le  U_i^{ts}(\bar{x})z^{ts}_i + \sum_{j\in \J}\sum_{k\in \K_j} \left[\bar{s}_{jk} (x^t_{jk} - \bar{x}^t_{jk})\right]^+.
\end{multline}
The latter cut is known as a generalized Benders' cut \citep{geoffrion1972}, and $\bar{s}_{jk} \in \partial U^{ts}_i(\bar{x})$ is any supergradient of $U^{ts}_i(x)$ at $\bar{x}$. In a similar spirit to \cite{fischetti2017}, we explain in the following how we can compute the values of $\bar{s}_{jk}$ using the Lagrangian function. 

To this end, $U^{ts}_i(x)$ can be bounded by the objective value of problem \eqref{BL1_OF3}-\eqref{BL1_sigmabinary}. Using \eqref{BL1_covering} to replace the values of $u^{ts}_{ij}$ in terms of $x$, the objective function \eqref{BL1_OF3} becomes:
\begin{equation} \label{eq:OB3_ureemplazada}
    \max_{\sigma^{ts}_i} \quad \sum_{j\in\J} \sum_{r\in\J} \lambda_{ir} \sum_{k\in \K_j}a_{ijk}^{ts}x_{jk}^t \sigma^{ts}_{ijr}
\end{equation}
We linearize the objective function defining a new set of assignment variables $\sigma^{ts}_{ijkr}:=x^t_{jk}\sigma^{ts}_{ijr}$ $\forall i\in I$, $t \in\T$, $s\in\S$, $j,r\in\J$, $k\in\K_j$. Intuitively, $\sigma^{ts}_{ijkr}=1$ if and only if $x^t_{jk}=1$ and $u^{ts}_{ij}=a^{ts}_{ijk}$ is the $r$-th greatest partial attraction for customer $i$ at time period $t$ in scenario $s$. Using these new set of variables, model \eqref{BL1_OF3}-\eqref{BL1_sigmabinary} is reformulated as the following LP (barring subscripts and superscripts $i$, $t$, $s$ to ease notation):
\begin{subequations} \label{SUB}
\begin{align}
    \max_{\sigma}\quad & \sum_{j \in \J}\sum_{r \in \J} \sum_{k \in \K_j}\lambda_r a_{jk} \sigma_{jkr} \label{SUB_OF}\\
    \text{s.t.}\quad & \dsum_{j \in \J} \sum_{k \in \K_j}\sigma_{jkr} \le 1,\quad \forall r \in \J,  \label{SUB_sumasigmajk}\\
    & \dsum_{r \in \J} \sum_{k \in \K_j} \sigma_{jkr} \le 1,\quad \forall j \in \J,  \label{SUB_sumasigmark}\\
    &\dsum_{r\in \J} \sigma_{jkr} \le x_{jk},\quad \forall j \in \J, k\in \K_j \label{SUB_sumasigmar} \\    
    & \sigma_{jkr} \ge 0, \quad \forall j,r\in \J, k\in\K_j. \label{SUB_csigmacontinua}
\end{align}
\end{subequations}
Note that this linearization has many more variables than the one proposed in Section~\ref{sec:MILP}. However, these variables are to be projected out, and this linearization has the advantage of having a much simpler structure, since it is a linear assignment problem. In the same spirit, observe also that constraints \eqref{SUB_sumasigmark} are dominated by \eqref{SUB_sumasigmar} because any feasible solution of (MP) satisfies $\sum_{k\in \K_j} x_{jk} \le 1$. However, we decided to keep them in the model to provide a solution of its dual easier to understand.

Let $(\bar{x},\bar{z},\bar{u})$ be a solution of the master problem \eqref{MasterProblem} and consider fixed customer $i$, time period $t$ and scenario $s$. Consider an optimal solution $\sigma_{jkr}^*$ of \eqref{SUB}, and let $\gamma_r^*$, $\delta_j^*$, $\eta_{jk}^*$ be nonnegative optimal dual variables for each set of constraints \eqref{SUB_sumasigmajk}-\eqref{SUB_sumasigmar}. Then the Lagrangian function of $U(x)$ at $\bar{x}$ in $\sigma_{jkr}^*$, $\gamma_r^*$, $\delta_j^*$, $\eta_{jk}^*$ is:
\begin{multline}
    \sum_{j \in \J}\sum_{r \in \J} \lambda_r a_{jk}\sigma^*_{jkr} + \sum_{r\in \J} \gamma^*_r (1 -\sum_{j\in \J}\sum_{k\in \K_j} \sigma^*_{jkr}) + \\
    + \sum_{j\in \J} \delta^*_j (1 -\sum_{r\in \J}\sum_{k\in \K_j} \sigma^*_{jkr}) + \sum_{j \in \J}\sum_{k \in \K_j} \eta^*_{jk} (\bar{x}_{jk} - \sum_{r\in \J} \sigma^*_{jkr}),
\end{multline}
\noindent so $\bar{s}_{jk}$ depends exclusively on the dual values $\eta_{jk}$: $\bar{s}_{jk} = \eta^*_{jk}$
and the generalized Benders' cut reads
\begin{equation} \label{linearOAcuts}
Tz \le \left(\sum_{j \in \J}\sum_{r \in \J} \lambda_r a_{jk}\sigma^*_{jkr}\right)z + \sum_{j\in \J}\sum_{k\in \K_j} \left[\eta^*_{jk} (x_{jk} - \bar{x}_{jk})\right]^+.
\end{equation}
Rather than solving the dual of problem \eqref{SUB}, we have derived specialized primal and dual algorithms to obtain the optimal values of the dual variables for any integer vector $\bar{x}$. In this way, we provide a fast inclusion of cuts that are numerically accurate. First, we introduce the dual problem of \eqref{SUB}:
\begin{subequations} \label{SUBDUAL}
\begin{align}
    \min\quad & \sum_{r \in \J} \gamma_r + \sum_{j\in \J} \delta_j + \sum_{j\in \J} \sum_{k \in \K_j} x_{jk} \eta_{jk} \label{SUBDUAL_OF}\\
    \text{s.t.}\quad &  \gamma_r + \delta_{j} + \eta_{jk} \ge \lambda_ra_{jk},\quad \forall j,r \in \J, k\in\K_j.\label{SUBDUAL_sigma} 
\end{align}
\end{subequations}
Next, we introduce the algorithms to solve \eqref{SUB} and \eqref{SUBDUAL} for fixed $i,t,s$.
\begin{algorithm}[H]
\caption{[Primal Algorithm]}\label{alg:BendersPrimal}
Input: Integer solution $(\bar{x}, \bar{z},\bar{u})$ of problem \eqref{MasterProblem}. \\
Output: Optimal solution $(\sigma_{jkr}^*)$  of problem \eqref{SUB}.
\begin{algorithmic}[1]
\Require Ordering $\tau: \J \rightarrow \J$ such that $\bar{u}_{\tau(1)} \ge \dots \ge \bar{u}_{\tau(|\J|)}$.
\For{$j \in \J$} 
    \For{$r \in \J$} 
        \If{$j == \tau(r)$}
            \State $\bar{k}_j \gets k$ if $\bar{x}_{jk}=1$, $\bar{k}_j \gets 0$ if $\sum_{k\in \K_j}\bar{x}_{jk}=0$
            \For{$k \in \K_j$}
                \If{$k == \bar{k}_j$}
                    \State $\sigma^*_{jkr} \gets 1$
                \Else
                    \State $\sigma^*_{jkr} \gets 0$
                \EndIf
            \EndFor            
        \EndIf
    \EndFor
\EndFor
\end{algorithmic}
\end{algorithm}

\begin{algorithm}[H]
\caption{[Dual Algorithm]}\label{alg:BendersDual}
Input: Integer solution $(\bar{x}, \bar{z}, \bar{u})$ of problem \eqref{MasterProblem}.  \\
Output: Optimal solution $({\gamma}_r^*, \delta_{j}^*, \eta_{jk}^*)$ of problem \eqref{SUBDUAL}.
\begin{algorithmic}[1]
\Require Ordering $\tau: \J \rightarrow \J$ such that $\bar{u}_{\tau(1)} \ge \dots \ge \bar{u}_{\tau(|\J|)}$.
\For{$r \in \J$} 
    \State $\gamma^*_r \gets \dsum_{r'=r}^{|\J|-1} \left(\lambda_{r'} - \lambda_{r'+1}\right) \bar{u}_{\tau(r')} + \lambda_{|\J|}\bar{u}_{\tau(|\J|)}$
\EndFor
\For{$j \in \J$} 
    \State $\delta^*_j \gets \dsum_{r' = \tau^{-1}(j)}^{|\J|-1} \lambda_{r'+1}\left( \bar{u}_{\tau(r')} - \bar{u}_{\tau(r'+1)}\right)$
\EndFor
\For{$j \in \J$}
    \State $\bar{k}_j \gets k$ if $\bar{x}_{jk}=1$, $\bar{k}_j \gets 0$ if $\sum_{k\in \K_j}\bar{x}_{jk}=0$
    \For{$k \in \K_j$}
        \If{$k \le \bar{k}_j$}  
            \State $\eta^*_{jk} \gets 0$
        \Else
            \State $r^*_{jk} \gets \min \left\{|\J|, \left\{r \in \J : \bar{u}_{\tau(r)} < a_{jk}  \right\} \right\}$
            \State $\eta^*_{jk} \gets \lambda_{r^*_{jk}}a_{jk} - \gamma^*_{r^*_{jk}} - \delta^*_j$ 
        \EndIf
    \EndFor
\EndFor
\end{algorithmic}
\end{algorithm}

\begin{theorem} \label{tma:AlgoritmoBenders}
Algorithm \eqref{alg:BendersPrimal} (resp.\ \eqref{alg:BendersDual}) provides an optimal solution of formulation \eqref{SUB} (resp.\ \eqref{SUBDUAL}) for a given integer vector $\bar{x}$. 
\end{theorem}
\begin{proof}
    The proof can be found in \ref{app:tmaAlgoritmoBenders}.
\end{proof}
Using the notation from Algorithms \ref{alg:BendersPrimal} and \ref{alg:BendersDual} and simplifying the value of $\eta^*$, the Benders' cut introduced for $i\in \I$, $t\in \T$, $s\in \S$, to cut out infeasible solutions of (MP) is:
\begin{equation} \label{eq:BendersCut}
    T_i^{ts} z^{ts}_i \le \left(\sum_{j \in \J} \lambda_{i\tau^{-1}(j)} \bar{u}^{ts}_{ij}\right) z^{ts}_i + \sum_{j\in \J}\sum_{k\in \K_j: \atop k>\bar{k}^t_j} \left(\lambda_{ir^*_{jk}} a^{ts}_{ijk} - \lambda_{i\tau^{-1}(j)}\bar{u}^{ts}_{ij} \right) x^t_{jk},
\end{equation}
\noindent where $\bar{k}^t_j=0$ if $\sum_{k\in \K_j} \bar{x}^t_{jk}=0$, and $\bar{k}^t_j$ is the unique $k$ such that $\bar{x}^t_{jk}=1$ otherwise (it coincides with the definition of $\bar{k}_j$ in Algorithms \ref{alg:BendersPrimal} and \ref{alg:BendersDual}). Note that expression \eqref{eq:BendersCut} coincides with \eqref{linearOAcuts}. Indeed, $\eta^*_{jk}\bar{x}_{jk}=0$, since $\bar{x}^t_{jk}=1 \Rightarrow \eta^*_{jk}=0$. Furthermore, $\eta^*_{jk}=0$ $\forall j\in \J$, $k\in \K_j$ such that $k\le \bar{k}_j$, and $\eta^*_{jk} = \lambda_{r^*_{jk}}a_{jk} - \gamma^*_{r^*_{jk}} - \delta^*_j$ is simplified in the second sum of the above expression for the remaining $k\in \K_j$.

\section{Computational study} \label{sec:EstudioComputacional}

In this section we report results from computational experiments that empirically show our contribution to the OCMCLP. 
We use two different data sets, one of a synthetic character, where the parameters are randomly generated, and another one based on the real data set provided by \citet{anjos2022arxiv}. The first one is designed to show and analyze the computational improvement implied by the methodology developed in the paper. The real data set is used to test the performance of our approach in a real application of the problem studied in the literature, specifically, the location of charging stations for electric vehicles. 
Throughout the section, we refer to these sets as \texttt{Syn-Data} and \texttt{Real-Data}, respectively.

In these experiments, we use \texttt{SL} to denote the MILP formulation \eqref{SL}, \texttt{VI} for the MILP formulation incorporating valid inequalities \eqref{ddvv_cotautilidad}, \eqref{ddvv_outlets} (and hence, with $\sigma \in [0,1]$ due to Proposition \ref{prop:asignacionrelajada}), \eqref{eq:SL_xcrecientedesagregadasVVII} and \eqref{eq:SL_wacotadasporyVVII} and the preprocessing due to Proposition \ref{prop:preproNoOrdenCompleto} from Section \ref{sec:VVII}, and \texttt{B} for the Benders' Decomposition approach presented in Section \ref{sec:Benders}.

All experiments are run on a Linux-based server with CPUs clocking at 2.6 GHz, 8 threads and 32 gigabytes of RAM. The models are coded in Python 3.7 and we used Gurobi 9.5 as optimization solver. 

The rest of the section is organized as follows: Section \ref{sec:data} defines the parameters and the size of the two data sets used for the experiments; Section \ref{sec:performance} shows the computational results of the proposed models and approaches; finally, Section \ref{sec:casestudy} presents the case study on the location of charging stations for electric vehicles.

\subsection{Data}\label{sec:data}
The parameters and sizes of the instances used for the simulations are given in the following subsections. However, here we include the definition of some parameters that apply to both \texttt{Syn-Data} and  \texttt{Real-Data}. 

Specifically, the model for the parameter $a_{ijk}^{ts}$ representing the partial attraction is:
$$
a_{ijk}^{ts} = \hat{a}_{ijk}^{t} + \bar{a}_{ij}^{t} + \epsilon^{ts}_{ijk},\; \forall t\in \T, s \in \S, i \in \I, j \in \J, k \in \K_j,
$$
where $\hat{a}_{ijk}^{t}$ is  associated to the type of station $k$ placed in $j$; $\bar{a}_{ij}^t$ is related to the features of location $j$; and $\epsilon_{ijk}^{ts}$ is the error associated to scenario $s$. This definition is consistent with that given by \citet{anjos2022arxiv}, and is useful for the \texttt{Real-Data} case study. 

Furthermore, we have tested instances with different $\boldsymbol{\lambda}$-vectors to compare their impact on the computational performance of the different models considered and the solutions obtained. Although they can differ, in these experiments we have taken the same values of the $\boldsymbol{\lambda}$-vectors for all $i\in \I$, $t \in \T$, $s \in \S$ of each instance. Hence, $\boldsymbol{\lambda}$ is a $|J|$-dimensional vector, and the different values considered are: 
\begin{description}
    \item[Type \texttt{C}:] $\boldsymbol{\lambda} = (1,0,\ldots,0)$. This type is the standard one in the MCLP literature, where the total level of attraction corresponds to the maximum attraction of any open facility.
    \item[Type \texttt{G}:] $\boldsymbol{\lambda} = (1,\frac{1}{9},\frac{1}{27},0,\ldots,0)$. This is based on the assumption that the weights associated to the partial attractions decrease following a geometric rule, and it takes into account a maximum of three facilities to calculate the total level of attraction.
    \item[Type \texttt{K}:] $\boldsymbol{\lambda} = (1,1,0,\ldots,0)$. It models the total level of attraction as an aggregation of the partial attractions of the two most attractive facilities.
    \item[Type \texttt{L}:] $\boldsymbol{\lambda}= (1,\frac{1}{2},0,\ldots,0)$. For this type, the weights are assumed to decrease in a linear fashion, and only the two best facilities are considered by the customer.
\end{description}

\subsubsection{Synthetic data}

As for the sizes of the instances of the set \texttt{Syn-Data}, we consider three time periods, $|\T| = 3$, and we test two sets of scenarios, $|\S| \in \{5,10\}$.
The coordinates of the customer classes are generated uniformly and randomly in the square $[0,1]\times [0,1]$ \citep[these kinds of sets are frequently used in the location literature, see, e.g.,][among others]{revelle2008solving,cordeau2019benders,lin2021,baldomero2022upgrading}. The sizes chosen for the set of  customer classes are $|\I|\in \{20,30,40,50\}$, and we generate five instances of each size, so 20 instances in total. In the case of the set of candidate locations of the facilities, we uniformly and randomly generate two sets with sizes $|\J|\in \{10,30\}$, although we only run instances with $|\J| \leq |\I|$. Finally, four types of facilities can be installed at each location, that is, $|K_j| = 4$ $\forall j\in \J$.

The weight of each class $n_i^t$,  $\forall i\in \I, t \in \T$, is given uniformly at random in the interval $[0,1]$. The threshold is set to the same value for each customer class, time period and scenario, and three values are considered: $T_i^{ts} \in \{10,12,15\}$. The maximum budget is set at the same value for all periods, $b^t \in\{5,10\}$. The costs of opening different facilities are equal for each $t\in \T$ and $j\in \J$, but vary with $k \in \K_j$ following the function $c_{jk}^t = k+3$.

Finally, we define the three values that form the partial attractions, namely,
\begin{description}
\item[$\hat{a}_{ijk}^{t}$:] This value is defined  for each $k \in \K_j$ as $\frac{k}{2}$.
\item[$\bar{a}_{ij}^{t}$:] We calculate all the distances among customer classes $i \in \I$ and candidate locations $j \in \J$, and distribute them into four groups (determined by the quantiles Q1, Q2, and Q3 of the computed distances). 
If the distance between a customer class and a candidate location is strictly below Q1, the value of the parameter is set to 8; if it is in (Q1,Q2], then it is set to 4; if it is in (Q2,Q3], then it is fixed to 2; and for the distances strictly above Q3, the assigned value is 0. 
\item[$\epsilon^{ts}_{ijk}$:] It is a random value that follows a standard normal distribution.
\end{description}

Thus, a total of 1680 instances are tested in our computational experiments with the data set \texttt{Syn-Data}. Any interested reader can replicate these experiments by finding the value of all parameters in our online GitHub repository \citep{GitHubOASYS}.

\subsubsection{Real data}
In our case study, we use a data set provided by \citet{anjos2022arxiv} based on the city of Trois-Rivières, Québec, which is divided into 317 zones. The authors consider the centroids of each zone as customer classes $i$ with weights $n^t_i$ equal to the 10\% of the population of each zone for every time period. A network is generated with a node for each customer class and edges linking adjacent zones. In this case, the Euclidean distance between each centroid is set as the length of the edge. Additionally, a subset of 30 locations among the centroids is chosen as the set of candidate locations for the installation of the facilities.

To create the instances, we consider subsets of their set of customer classes of three different sizes, namely $|\I| \in \{100,200,317\}$, and we do likewise for the set of potential locations of facilities, $|\J| \in \{10,20,30\}$. For the subsets of $\I$, we take the classes with larger weights. As for the subsets of $\J$, we select the first 10 and 20 of the total set. Finally, we consider their short and long span, $|\T| \in \{4,10\}$, and we set the number of scenarios to $|\S| = 5$.

Moreover, \citet{anjos2022arxiv} define different types of charging stations according to the number of outlets they contain, i.e., the capacity of the station to charge vehicles simultaneously. This definition is consistent with the assumptions made in the paper, in the sense that the customer's attraction increases as the number of outlets in a facility does. We maintain the maximum number of outlets per station proposed by the authors, $|\K_j| = 6, \forall j \in \J$. The authors define separate costs for opening a facility and for increasing its number of outlets. In our case, these costs are incorporated in the value of the parameter $c$ as $c_{jk}^t = 100 + 50k$. The total budget for all $t \in \T$ is fixed to $b^t = 400$. In addition to the value of the threshold considered by the authors, $T_i^{ts}=4.5$, we also include instances with $T_i^{ts}=9$ $\forall i\in\I$, $t\in\T$, $s\in\S$. Finally, to define the partial attractions we also follow the authors' design:
\begin{align*}
\hat{a}_{ijk}^{t} & = 0.281k,\\ 
\bar{a}_{ij}^{t} & = 1.638 - 0.63d_{ij},\\
\epsilon^{ts}_{ijk} &= FT\xi^s + \zeta^s,
\end{align*}
\noindent for all $t\in \T, s \in \S, i \in \I, j \in \J, k \in \K_j$, and $d_{ij}$ representing the distance of the shortest path between $i$ and $j$. In the above, $F$ is a factor loading matrix, $T$ is a diagonal matrix, $\xi^s$ is a vector of IID random terms from a normal distribution with location zero and a scale of one, and $\zeta^s$ is a vector of IID random terms from a Gumbel distribution with location of zero and a scale of three. However, for our real-world case study, we adopt the approach proposed by the authors, where customers consider only the facilities located within a radius of ten kilometers, and they have no attraction for those outside the radius,  (i.e., $u_{ij}^{ts}=0$ for any pair $(i,j)$ whose distance is superior to ten kilometers). For more information about the definition of these parameters, see \citet{anjos2022arxiv}. 

Following the same scheme as in this work, we solve 20 independent and different instances on this data set by modifying only the random vectors $\xi$ and $\zeta$. As a result, the set \texttt{Real-Data} has a total of 2880 instances. These instances are also available in our GitHub repository \citep{GitHubOASYS}.

\subsection{Study on computational performance} \label{sec:performance}

This section is devoted to show the computational performance of the solution approaches described throughout the paper on a synthetic data set (\texttt{Syn-Data}), and a time limit of 1 hour (3600 seconds) is established for each approach.

First, Figure \ref{fig:performanceprofile} provides two plots that effectively compare the performance of the proposed approaches in this paper. In both plots, the solid line represents the MILP model (\texttt{SL}); the dashed line, the MILP model with valid inequalities (\texttt{VI}); and the dotted line, the Benders' decomposition based approach (\texttt{B}).
Figure \ref{fig:performancetime} depicts a performance profile of the percentage of instances solved to optimality within a computational time in seconds. A point in the figure with coordinates $(x,y)$ indicates that for $y\%$ of the instances, the instance was solved in less than $x$ seconds. 
It is noticeable that \texttt{B} outperforms the others, that is, the number of instances solved by \texttt{B}  before the time limit (around a 70\%) is the highest out of the three models. On the other hand, \texttt{SL} solves around a 25\% of the instances (the smallest ones) in less computational time compared to the other two models, although it is only able to solve to optimality about 35\% of the instances proposed. 
Similarly, Figure \ref{fig:performancegap} shows a performance profile of the percentage of instances with respect to the MIP Gap after one hour of computation time. Thus, a point in the graph with coordinates $(x,y)$ indicates that for $y\%$ of instances, the MIP gap is less than $x\%$. Clearly, the instances solved to optimality by a certain model have a MIP gap of $0\%$. Here we see the clear improvement of \texttt{B} with respect to the two models presented: It solves more instances to optimality and when the time limit is reached, the gap is much lower. This follows from the fact that, with a $20\%$ gap, we have less than $50\%$ of the instances for \texttt{SL}, while for \texttt{VI} it is $60\%$ and more than $75\%$ for \texttt{B}.
Note that there are a few instances that end up with a GAP greater than 100\%, but we did not include these outliers in the performance profile.

\begin{figure}[H]
\centering
\begin{subfigure}[b]{0.49\textwidth}
\centering
\includegraphics[scale = 0.65]{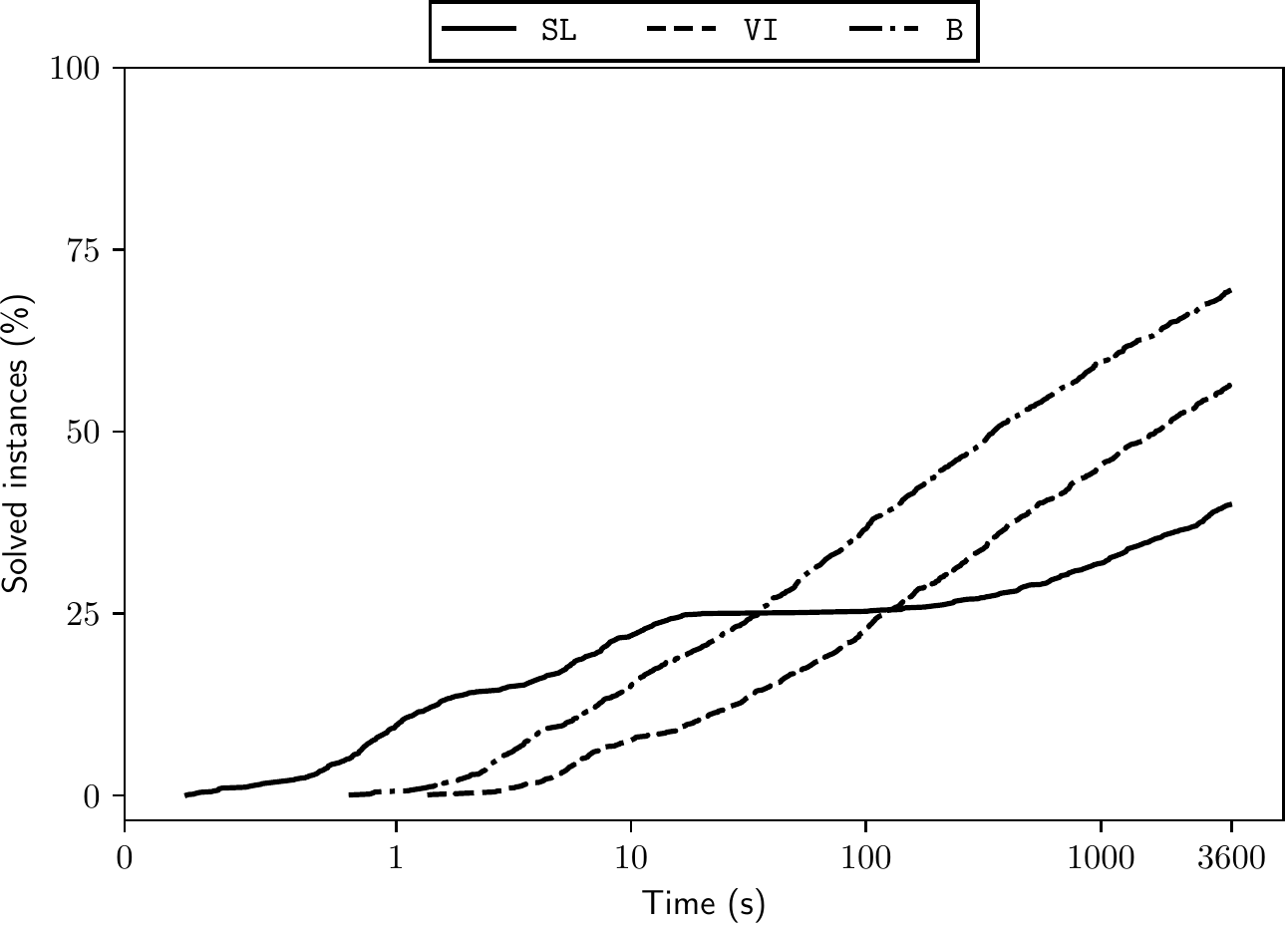}   
\caption{Performance profile of the percentage of solved instances as a function of time (using a logarithmic scale).}\label{fig:performancetime}
\end{subfigure}\hfill
\begin{subfigure}[b]{0.49\textwidth}
\centering
\includegraphics[scale = 0.65]{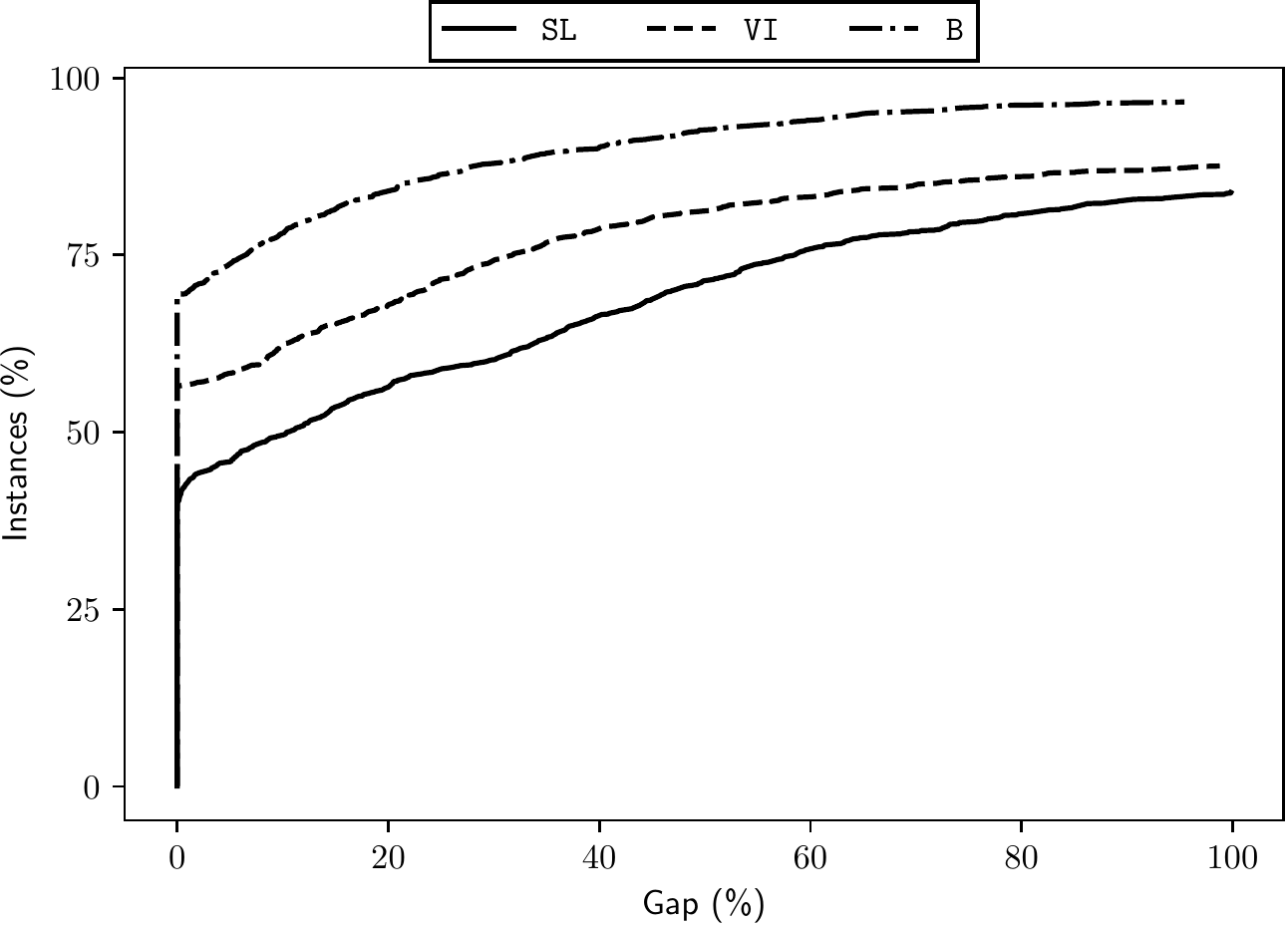}   
\caption{Performance profile of the percentage of instances as a function of the remaining MIP gap after a time limit of one hour.}
\label{fig:performancegap}
\end{subfigure}
\caption{Performance profiles.}
\label{fig:performanceprofile}
\end{figure}

With these plots, we illustrate that the solution approaches \texttt{VI} and \texttt{B} presented in this paper improve the initial formulation \texttt{SL}. Indeed, approach \texttt{B}  consistently outperforms the others, solving more instances in less time and returning smaller MIP gaps when it is not capable of solving the instance. 

The difference in performance associated to the size of the instances can be seen when we break down the solutions depending on the parameter values. In this line, the median computational time in seconds of the instances that finished within the time limit (\texttt{Time[s]}), the number of solved instances, and the median of the covered customers by period ($t$) for the instances solved with \texttt{B} are collated in Table \ref{t:lambopt} for various values of $\boldsymbol{\lambda}$ (types \texttt{C}, \texttt{G}, \texttt{K} and \texttt{L}) and different threshold values ($T$). There are 140 instances summarized per row.
In this table and the following ones, we have used the median of the data instead of the average to avoid the influence that the outliers may cause.

For the case when $\boldsymbol{\lambda}$ is set to \texttt{C}, which considers only the most useful facility, the \texttt{SL} model solves 100\% of the instances with very low computational times. However, for the rest of the $\boldsymbol{\lambda}$-vectors (where more than one facility is taken into account for the calculation of the total level of attraction), the \texttt{SL} model performs worse, giving lower percentages of solved instances and higher computational times compared to \texttt{VI} and \texttt{B}. In contrast, approach \texttt{B} shows a better computational performance in most cases, solving more instances and taking less time compared to \texttt{VI}. This suggests that approach \texttt{B} is more effective in handling the calculation of the total level of attraction when multiple facilities are involved. Thus, it can potentially offer improved computational performance when solving problems with general or varied $\boldsymbol{\lambda}$-vectors.

Note that the adjustment of the values given to the threshold and the $\boldsymbol{\lambda}$-vector has been made so that the percentage of customers covered ranges from 0\% to almost 100\%. This illustrates how a slight modification in these values may have a significant impact on the overall solutions reported in terms of covered customers. In turn, the selection of parameters affects the computational performance: extreme cases where either everyone or no one is covered are easier to solve than intermediate settings. The reason is that, in the extreme cases, the threshold value is so low or so high that the customers' decision is very little influenced by the location of the stations.

\begin{table}[H]
\begin{center}
    \begin{tabular}{ccrrrrrrcrrr}
    \toprule
  $T$ & $\boldsymbol{\lambda}$ & \multicolumn{3}{c}{\texttt{Time[s]}} & \multicolumn{3}{c}{ \texttt{\# Solved}} & & \multicolumn{3}{c}{\texttt{Cap. customer[\%]}}\\
 \cmidrule(lr){3-5}\cmidrule(lr){6-8} \cmidrule(lr){10-12}
 &  & \texttt{SL} &  \texttt{VI} & \texttt{B} & \texttt{SL} & \texttt{VI} & \texttt{B} & & $t=1$ & $t=2$ & $t=3$\\
  \hline
10 & \texttt{C} & 1.8 & 85.1 & 16.5 & 140 & 102 & 113 & & 20.0 & 38.2 & 52.2\\
   & \texttt{G} & 1355.2 & 919.1 & 331.1 & 20 & 54 & 82 & & 35.1 & 76.7 & 90.1\\
 & \texttt{K} & 1046.2 & 167.9 & 65.5 & 54 & 101 & 138 & & 46.5 & 91.9 & 98.4\\
 & \texttt{L} & 1089.4 & 495.3 & 197.6 & 48 & 75 & 119 & & 38.5 & 82.2 & 93.1\\
12 & \texttt{C} & 1.5 & 20.2 & 9.3 & 140 & 140 & 138 & & 1.1 & 2.9 & 4.2\\
 & \texttt{G} & 2359.8 & 1723.2 & 623.4 & 3 & 20 & 45 & & 14.3 & 46.4 & 67.0\\
 & \texttt{K} & 1272.6 & 260.9 & 177.6 & 32 & 88 & 118 & & 33.7 & 81.8 & 93.8\\
 & \texttt{L} & 1036.2 & 1029.8 & 348.1 & 12 & 45 & 66 & & 19.8 & 57.0 & 76.1\\
15 & \texttt{C} & 1.3 & 8.9 & 7.5 & 140 & 140 & 140 & & 0.0 & 0.0 & 0.0 \\
 & \texttt{G} & 1867.8 & 515.3 & 185.1 & 26 & 59 & 68 & & 0.8 & 8.5 & 18.1\\
 & \texttt{K} & 1756.9 & 844.9 & 260.5 & 8 & 66 & 66 & & 19.1 & 57.2 & 76.8\\
 & \texttt{L} & 671.3 & 281.3 & 225.3 & 48 & 59 & 74 & & 5.7 & 19.0 & 31.8\\
    \bottomrule
\end{tabular}
\caption{\texttt{Syn-Data}. The total number of instances per row is 140. Time limit equal to 3600 seconds. Median computational time of solved instances in seconds (\texttt{Time[s]}), number of instances solved by the different models for each threshold and $\boldsymbol{\lambda}$-value, and median number of covered customer classes (\texttt{Cap. customer[\%]}) per time period for each threshold and $\boldsymbol{\lambda}$-value.} 
\label{t:lambopt}
\end{center}
\end{table}

Table \ref{t:complexity} provides insight into the computational complexity of the problem for different number of scenarios ($|\S|$), instance sizes ($|\I|$), $\boldsymbol{\lambda}$-values, and the sizes of the set of candidate locations for the facilities ($|\J|$). For the sake of clarity, we include here only two data sets, $|\I| \in \{30, 50\}$. To see the results regarding the rest of the sizes, we refer the reader to \ref{app:syn}.
This table shows the median computational time in seconds only for the instances that have been solved to optimality within one hour (\texttt{Time[s]}), the number of solved instances between parentheses (\texttt{\# Solved}), and two GAP values for the unsolved instances: \texttt{MIPGap} $= \frac{|z_{bb}-z_{P}|}{z_{P}}\cdot 100$, and \texttt{FGap} $= \frac{|z_{bb} - z_{bP}|}{z_{bP}}\cdot 100$, where $z_{bb}$ is the best upper bound, $z_{P}$ is the incumbent value (i.e., the current best primal objective bound), and $z_{bP}$ is the best incumbent offered by any of the three approaches.
When none of the instances are solved within the time limit, we have written \texttt{TL} in the time column. There are 30 instances summarized per row. 

It is observed that approach \texttt{SL} struggles to solve instances with five scenarios, even with small-sized instances like those with 30 customer classes with a success rate of only 27.2\% for the $\boldsymbol{\lambda}$-values \texttt{G}, \texttt{K} and \texttt{L}. On the contrary, approach \texttt{B} performs well even with an increased number of customer classes. For instance, with 50 customer classes, a 63.3\% of the instances are solved for the same $\boldsymbol{\lambda}$-values \texttt{G}, \texttt{K} and \texttt{L}. 
For the simplest case $\boldsymbol{\lambda}$=\texttt{C}, it is observed that the inclusion of valid inequalities (approach \texttt{VI}) to the MILP model is not beneficial, since for $|\J|=30$ and for both values of $|\I|$, \texttt{SL} solves the 30 instances proposed while \texttt{VI} only solves 25. This suggests that the valid inequalities are of no use when a single station is taken into account in the computation of the total level of attraction. However, for the rest of the $\boldsymbol{\lambda}$-vectors \texttt{VI} clearly outperforms \texttt{SL}, (although \texttt{B} remains unbeaten).    

If we examine the results for ten scenarios, we observe that the slight increase in the number of scenarios already entails a deterioration in the computational performance of all the methods. This fact is particularly evident for specific $\boldsymbol{\lambda}$-values. For instance, approach \texttt{B} solves 22 (resp.\ 18) instances with five scenarios, $\boldsymbol{\lambda}$ set to \texttt{K}, 30 candidate facilities and 30 (resp.\ 50) customer classes. When the number of scenarios is increased to ten, the number of solved instances goes down to 14 (resp.\ 12). The case of $\boldsymbol{\lambda}$ set to \texttt{G} is more extreme, since none of the approaches is able to solve any instance for 30 facilities, even with 30 customer classes. 

\begin{table}[H]
\begin{center}
\small
    \begin{tabular}{cccrrrrrrrrrr}
    \toprule
$|\S|$ & $|\I|$ & $\boldsymbol{\lambda}$ & $|\J|$ & \multicolumn{3}{c}{\texttt{Time[s] (\#Solved)}} &  \multicolumn{3}{c}{\texttt{MIPGap[\%]}} & \multicolumn{3}{c}{\texttt{FGap[\%]}} \\ 
\cmidrule(lr){5-7}\cmidrule(lr){8-10}\cmidrule(lr){11-13}
&  &  &  & \texttt{SL} & \texttt{VI} & \texttt{B} & \texttt{SL} & \texttt{VI} & \texttt{B} & \texttt{SL} & \texttt{VI} & \texttt{B} \\
\hline
5 & 30 & \texttt{C} & 10 & 0.4 (30) & 6.5 (30)  & 2.5 (30)  & 0.0 & 0.0 & 0.0 & 0.0 & 0.0 & 0.0 \\
              &  &  & 30 & 3.6 (30) & 33.9 (25) & 17.5 (30)  & 0.0 & 22.0 & 0.0 & 0.0 & 22.0 & 0.0 \\
    &  & \texttt{G} & 10 & 1775.2 (11) & 652.2 (22) & 185.1 (29) & 35.6 & 13.3 & 1.2 & 35.5 & 13.2 & 1.2 \\
              &  &  & 30 & \texttt{TL} (0) & \texttt{TL} (0) & 1151.0 (6) & 104.5 & 77.6 & 50.1 & 94.9 & 57.5 & 50.1 \\
    &  & \texttt{K} & 10 & 1125.2 (20) & 86.4 (30) & 22.9 (30) & 22.9 & 0.0 & 0.0 & 22.9 & 0.0 & 0.0 \\
              &  &  & 30 & \texttt{TL} (0) & 1659.5 (18) & 160.6 (22) & 33.9 & 19.1 & 10.3 & 32.0 & 18.1 & 10.2 \\
    &  & \texttt{L} & 10 & 441.2 (18) & 351.0 (28) & 72.1 (30) & 12.0 & 8.2 & 0.0 & 12.0 & 8.2 & 0.0 \\
              &  &  & 30 & \texttt{TL} (0) & \texttt{TL} (0) & 378.7 (12) & 74.4 & 44.0 & 25.3 & 69.2 & 43.1 & 25.2 \\
& 50 & \texttt{C} & 10 & 0.9 (30) & 9.3 (30) & 4.0 (30) & 0.0 & 0.0 & 0.0 & 0.0 & 0.0 & 0.0 \\
            &  &  & 30 & 7.5 (30) & 99.0 (25) & 52.1 (27) & 0.0 & 36.2 & 2.9 & 0.0 & 36.2 & 2.9 \\
    &  & \texttt{G} & 10 & 3089.7 (4) & 1057.2 (16) & 577.7 (27) & 53.5 & 32.9 & 2.9 & 50.3 & 30.4 & 2.9 \\
              &  &  & 30 & \texttt{TL} (0) & \texttt{TL} (0) & 1709.7 (3) & 123.0 & 345.4 & 50.4 & 94.1 & 115 & 50.4 \\
    &  & \texttt{K} & 10 & 1507.9 (7) & 267.8 (30) & 145.3 (30) & 8.4 & 0.0 & 0.0 & 8.4 & 0.0 & 0.0 \\
              &  &  & 30 & \texttt{TL} (0) & 2634.0 (3) & 591.1 (18) & 39.0 & 17.9 & 11.1 & 37.0 & 17.9 & 11.1 \\
    &  & \texttt{L} & 10 & 1147.3 (11) & 638.9 (21) & 197.6 (29) & 23.9 & 19.7 & 4.0 & 23.9 & 19.7 & 4.0 \\
              &  &  & 30 & \texttt{TL} (0) & \texttt{TL} (0) & 1713.3 (7) & 76.9 & 52.3 & 31.1 & 63.3 & 48.6 & 31.1 \\
10 & 30 & \texttt{C} & 10 & 0.9 (30) & 11.7 (30) & 6.3 (30) & 0.0 & 0.0 & 0.0 & 0.0 & 0.0 & 0.0 \\
               &  &  & 30 & 6.7 (30) & 112.0 (25) & 49.2 (25) & 0.0 & 27.2 & 2.0 & 0.0 & 27.2 & 2.0 \\
    &  & \texttt{G} & 10 & 1890.4 (1) & 1210.5 (13) & 671.5 (20) & 56.3 & 27.3 & 14.8 & 51.1 & 27.3 & 14.8 \\
              &  &  & 30 & \texttt{TL} (0) & \texttt{TL} (0) & \texttt{TL} (0) & 126.7 & 219.0 & 47.0 & 103.8 & 39.8 & 47.0 \\
    &  & \texttt{K} & 10 & 2539.9 (9) & 355.5 (30) & 174.6 (28) & 14.9 & 0.0 & 14.5 & 13.8 & 0.0 & 14.3 \\
              &  &  & 30 & \texttt{TL} (0) & 3365.0 (2) & 475.2 (14) & 38.9 & 14.4 & 10.3 & 33.4 & 14.4 & 10.2 \\
    &  & \texttt{L} & 10 & 1548.3 (11) & 791.1 (18) & 665.3 (28) & 23.4 & 21.5 & 6.7 & 23.4 & 21.5 & 6.7 \\
              &  &  & 30 & \texttt{TL} (0) & \texttt{TL} (0) & 1151.3 (5) & 86.1 & 50.7 & 31.8 & 74.9 & 34.5 & 31.8 \\
& 50 & \texttt{C} & 10 & 1.6 (30) & 30.1 (30) & 14.7 (30) & 0.0 & 0.0 & 0.0 & 0.0 & 0.0 & 0.0 \\
            &  &  & 30 & 13.6 (30) & 215.2 (20) & 73.0 (18) & 0.0 & 44.3 & 8.6 & 0.0 & 44.2 & 8.6 \\
    &  & \texttt{G} & 10 & \texttt{TL} (0) & 1597.1 (5) & 338.6 (10) & 56.6 & 33.2 & 20.3 & 54.4 & 31.8 & 20.3 \\
              &  &  & 30 & \texttt{TL} (0) & \texttt{TL} (0) & \texttt{TL} (0) & 166.8 & 404.1 & 67.1 & 114.5 & 46.0 & 67.1 \\
    &  & \texttt{K} & 10 & 3128.0 (1) & 656.3 (22) & 299.3 (21) & 18.5 & 12.7 & 6.8 & 18.5 & 11.7 & 6.8 \\
              &  &  & 30 & \texttt{TL} (0) & \texttt{TL} (0) & 1184.4 (12) & 59.7 & 23.8 & 17.9 & 39.0 & 16.9 & 17.5 \\
    &  & \texttt{L} & 10 & 2723.4 (3) & 1421.0 (13) & 1191.6 (17) & 31.3 & 22.6 & 7.7 & 29.4 & 21.7 & 7.7 \\
              &  &  & 30 & \texttt{TL} (0) & \texttt{TL} (0) & 2659.9 (4) & 101.5 & 61.8 & 33.6 & 78.8 & 31.3 & 33.6 \\
    \bottomrule
\end{tabular}
\caption{\texttt{Syn-Data}. The total number of instances per row is 30. Time limit equal to 3600 seconds. Median computational time of solved instances in seconds (\texttt{Time[s]}), number of solved instances (\texttt{\#Solved}), median Gap between the incumbent solution and the best bound (\texttt{MIPGap}), and median Gap between the best solution found among the models and the model's best bound (\texttt{FGap}).} 
\label{t:complexity}
\end{center}
\end{table}

In terms of computational times, in the majority of the combinations of parameters the median  is under 1800s. Approach \texttt{B} provides the best solution times, with reductions of up to one order of magnitude with respect to \texttt{SL} or half the time compared to \texttt{VI} (for $|\I|$=50, $|\J|$=30 and $\boldsymbol{\lambda}$=\texttt{G}).

For the instances that are not solved to optimality within the time limit, approach \texttt{B} achieves low MIPGap values, typically remaining below a 10.6\% of gap for five scenarios and 17.0\% for ten scenarios, based on the instances analyzed. This implies that approach \texttt{B} effectively solves larger instances and can achieve good optimality gaps indicative of near-optimal solutions.

As for the instances with \texttt{MIPGap} higher than 100\%, this is due to the incumbent solution provided at the time limit being very close to zero. This results in a high percentage gap, but it is interesting to observe that the upper bound is similar among the three models even in cases where the incumbent is close to zero. This is shown in the column \texttt{FGap}, where we compare the bound of each model with respect to the best incumbent found. Thus, the models provide meaningful bounds even when the solution is not near-optimal within the time limit.

These findings indicate that approach \texttt{B} is more suitable for solving larger instances with five or ten scenarios, even with an increased number of customer classes  and candidate locations, and can achieve good solution quality with low optimality gaps. Therefore, approach \texttt{B} may be preferred for practical applications with larger instances, while approach \texttt{SL}  may be suitable for smaller instances with fewer scenarios, or when the calculation of the total level of attraction only involves a single facility (i.e., when $\boldsymbol{\lambda}=\texttt{C}$).

\subsection{Case study on the installation of charging stations for electric vehicles}\label{sec:casestudy}

We conclude this section by presenting a case study for increasing electric vehicle adoption through the placement of new charging stations in the city of Trois-Rivières, Québec, using the instances employed by \citet{anjos2022arxiv}. In this section, in light of the results discussed in the previous section, we use the approach with the best performance for each $\boldsymbol{\lambda}$-vector, namely, approach \texttt{SL} for \texttt{C}, and approach \texttt{B} for the rest of them. In this case study, the time limit is set to 8 hours (28800 seconds).

\begin{table}[H]
\begin{center}
    \begin{tabular}{ccrrrr}
    \toprule
    $|\J|$ & $|\I|$ & \multicolumn{2}{c}{$T = 4.5$} & \multicolumn{2}{c}{$T = 9$} \\ \cmidrule(lr){3-4}\cmidrule(lr){5-6}
 &  & \texttt{Time[s]} \texttt{(\# Solved)} & \texttt{MIPGap[\%]} & \texttt{Time[s]} \texttt{(\# Solved)} & \texttt{MIPGap[\%]} \\
\hline
10 & 100 & 261.9 (80) & 0.0 & 3963.4 (65) & 1.5 \\
 & 200 & 5076.2 (57) & 1.0 & 8.3 (20) & 4.9 \\
 & 317 & 53.0 (35) & 1.5 & 14.0 (20) & 6.1 \\
20 & 100 & 3752.6 (48) & 1.5 & 8.8 (20) & 9.5 \\
 & 200 & 122.5 (20) & 6.4 & 35.3 (20) & 16.0 \\
 & 317 & 455.3 (20) & 7.8 & 87.1 (20) & 17.8 \\
30 & 100 & 30.5 (21) & 2.9 & 16.5 (20) & 13.9 \\
 & 200 & 1260.6 (20) & 9.0 & 117.8 (20) & 20.8 \\
 & 317 & 8716.9 (20) & 10.1 & 1205.2 (20) & 22.4 \\
    \bottomrule
\end{tabular}
\caption{\texttt{Real-Data} Computational results with $|\T| = 4$. The total number of instances per row is 80. Time limit equal to 28800 seconds (8 hours). Median computational time of solved instances in seconds (\texttt{Time[s]}), number of solved instances (\texttt{\#Solved}), and median gap using the incumbent solution and the best bound (\texttt{MIPGap}).} 
\label{t:anjos-computational}
\end{center}
\end{table}

Table \ref{t:anjos-computational} summarizes the computational experience for larger instances, for different sizes of the set of customers ($|\I|$) and potential locations ($|\J|$), and for different values of the threshold ($T$). Each row summarizes 80 instances, that is, 20 instances for each $\boldsymbol{\lambda}$-value. The table shows the median computational time of solved instances in seconds (\texttt{Time[s]}), the total number of instances solved (\texttt{\#Solved}), and the median gap using the incumbent solution found and the best bound (\texttt{MIPGap}) when $|\T| = 4$. We refer the reader to \ref{app:real} for a summary of our computational results for $|\T | = 10$.

This table illustrates how the difficulty of the problem scales up when the size of the instance increases (compared to the sizes considered in \texttt{Syn-Data}). In fact, when only 20 instances out of 80 are solved, these 20 instances are the simplest ones, i.e., the ones with $\boldsymbol{\lambda} = \texttt{C}$. It is noteworthy that, for an threshold of 4.5, we were able to solve many instances in a relatively shorter time period despite the larger problem sizes of this case study (compared to the previous one). This is due to the fact that customers consider only facilities within a 10 km radius. As a result, some customers consider only one facility, while others consider the entire set of facilities. This modeling assumption is aligned with the realistic setting where customers prioritize facilities in the proximity to their location, and disregard those that are further away.

Next we examine how different $\boldsymbol{\lambda}$-values influence the optimal location of charging stations. For this purpose, we compare the solutions obtained assuming the cooperative $\boldsymbol{\lambda}$-values, \texttt{G}, \texttt{K} and \texttt{L}, with those given by the standard \texttt{C} from the literature. 
In Table \ref{t:deviation}, we depict the regret of locating the stations assuming $\boldsymbol{\lambda}$=\texttt{C} when another $\boldsymbol{\lambda}$-value should have been assumed instead.
This regret is calculated as a percentage deviation, $\texttt{deviation}_{\boldsymbol{\lambda}} = \frac{f_{\lambda}(x) - f_\lambda(x^*_\texttt{C})}{f_{\lambda}(x)}\cdot 100$, where $f_{\lambda}(x)$ represents the best objective value found when $\boldsymbol{\lambda}$ is used, and $f_\lambda(x^*_\texttt{C})$ represents the objective value when $\boldsymbol{\lambda}$ is used fixing the location decision to the one found by the \texttt{C}-model.

The results of Table \ref{t:deviation} are the average of these deviations 
and are given by instance size ($|\J|$ and $|\I|$), by lambda value ($\boldsymbol{\lambda} \in \{\texttt{G}, \texttt{K}, \texttt{L}\}$) and by threshold ($T$), only for $|\T| = 4$ (for the long span, see \ref{app:real}). 
For these averages of regrets, we have included only the instances such that \texttt{C} is solved to optimality, and averaged all the non negative regrets (since the regret is in fact non negative if the solutions reported are optimal). We also want to point out that, despite the number of instances not solved for the cooperative $\boldsymbol{\lambda}$-vectors, only around a 7.7\% of the regrets are negative, so the immense majority of the solutions reported represent an improvement over the standard case $\boldsymbol{\lambda}=\texttt{C}$.

This table confirms that assuming a cooperative approach in facility placement decisions has a significant impact on the resulting location decisions. If considering a cooperative setting for the total level of attraction did not affect the location of the stations, the deviations in the objective value would have been zero. However, as shown in Table \ref{t:deviation}, there are noticeable deviations when using cooperative models, such as a 3\% deviation in one case, even in cases where optimality is not achieved. When the threshold value is more \textit{challenging} and the customers' decision rule is cooperative, a planning that takes into account their cooperative behavior is more important. This is clear when comparing the regrets obtained for  $T = 4.5$ and $T = 9$.

\begin{table}[H]
\begin{center}
    \begin{tabular}{ccrrrrrr}
    \toprule
$|\J|$ & $|\I|$ & \multicolumn{3}{c}{$T = 4.5$} & \multicolumn{3}{c}{$T = 9$} \\ \cmidrule(lr){3-5}\cmidrule(lr){6-8}
 &  & \multicolumn{1}{c}{\texttt{G}} & \multicolumn{1}{c}{\texttt{K}} & \multicolumn{1}{c}{\texttt{L}} & \multicolumn{1}{c}{\texttt{G}} & \multicolumn{1}{c}{\texttt{K}} & \multicolumn{1}{c}{\texttt{L}} \\
 \hline
10 & 100 & 0.7 & 1.1 & 0.8 & 1.9 & 2.2 & 2.2 \\
   & 200 & 0.5 & 0.6 & 0.6 & 1.2 & 1.3 & 1.3 \\
   & 317 & 0.4 & 0.5 & 0.4 & 0.9 & 1.1 & 1.0 \\
20 & 100 & 0.5 & 0.8 & 0.6 & 1.6 & 2.5 & 2.0 \\
   & 200 & 0.4 & 0.7 & 0.5 & 1.0 & 1.6 & 1.4 \\
   & 317 & 0.2 & 0.4 & 0.3 & 0.5 & 1.2 & 0.6 \\
30 & 100 & 0.5 & 0.9 & 0.6 & 2.3 & 3.2 & 2.9 \\
   & 200 & 0.5 & 0.8 & 0.6 & 0.8 & 1.6 & 1.2 \\
   & 317 & 0.3 & 0.5 & 0.3 & 0.6 & 1.4 & 0.8 \\
    \bottomrule
\end{tabular}
\caption{\texttt{Real-Data} with $|\T| = 4$. The total number of instances per row is 20. Instances solved with \texttt{B}. Time limit equal to 28800 seconds (8 hours). Average of the regret between \texttt{C} and other $\boldsymbol{\lambda}$-values.} 
\label{t:deviation}
\end{center}
\end{table}

Finally, we conclude by showing an example in Figure \ref{fig:ex_TR} of how the distribution of the stations changes for different $\boldsymbol{\lambda}$-values. In the figure, we present a specific instance that is solved to optimality, with the following parameters: $|\I| = 317, |\J| = 10, |\T| = 4, |\S| = 5,$ and $T = 4.5$. The figure illustrates the last period, where the black dots represent the 317 centroids of each subregion in Trois-Rivières, Québec. The red stars represent the open charging stations, while the yellow stars depict the stations that remained closed. The size variations among the stars indicate different types: larger stars represent charging stations with a greater number of outlets. 

Upon closer inspection of the figure, when $\boldsymbol{\lambda} = \texttt{C}$, all the stations are open, with only one of them having a greater number of outlets. This proves that for a threshold value of 4.5, distributing the facilities throughout the region leads to a better solution. For other $\boldsymbol{\lambda}$-values, multiple facilities are considered when calculating the total level of attraction. The case \texttt{G} is similar to \texttt{C} because the weights assigned by customers to their second and third options are very low, so the model tends to favor again solutions where the stations are distributed throughout the territory.

The cases with $\boldsymbol{\lambda}$ equal to \texttt{K} and \texttt{L}, where two facilities are considered to compute the covering, result in similar solutions: they are the only two cases in which a station remains closed. However, the distinction lies in the fact that for $\boldsymbol{\lambda} = \texttt{K}$, one station with 4 outlets and two stations with 2 outlets are open. Conversely, when $\boldsymbol{\lambda} = \texttt{L}$, the solution consists in opening one station with 6 outlets while the rest only have 1 outlet. This difference occurs because for $\boldsymbol{\lambda}$=\texttt{L} the customers attribute more importance to their first option than to their second one. As a result, the model favors a more concentrated distribution by opening more attractive stations, i.e., stations with more outlets.

\begin{figure}[H]
\centering
\begin{subfigure}[b]{0.49\textwidth}
\centering
\fbox{\includegraphics[scale = 0.4]{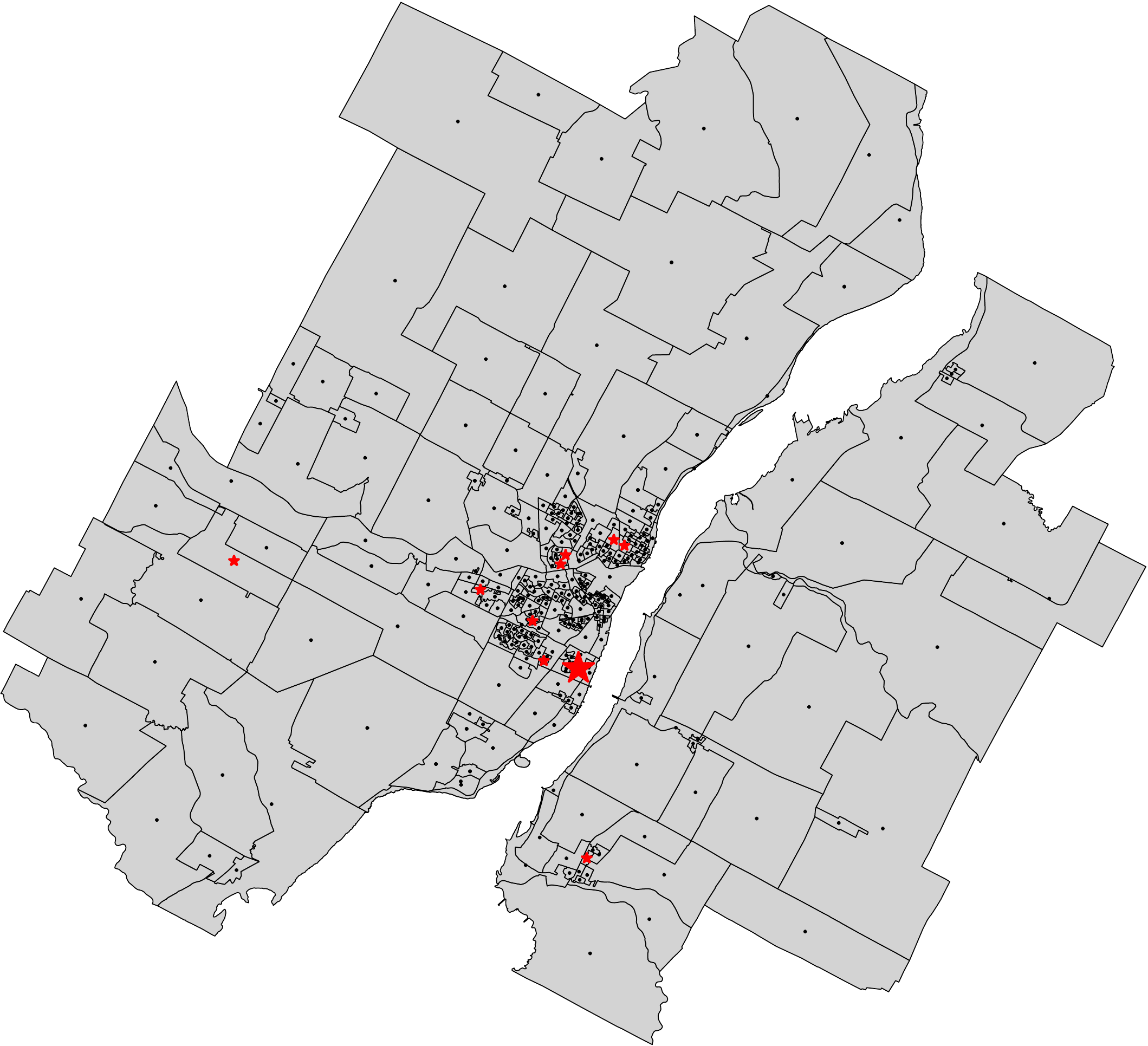}}
\caption{$\boldsymbol{\lambda}$ = \texttt{C}}\label{fig:TR_C}
\end{subfigure}\hfill
\begin{subfigure}[b]{0.49\textwidth}
\centering
\fbox{\includegraphics[scale = 0.4]{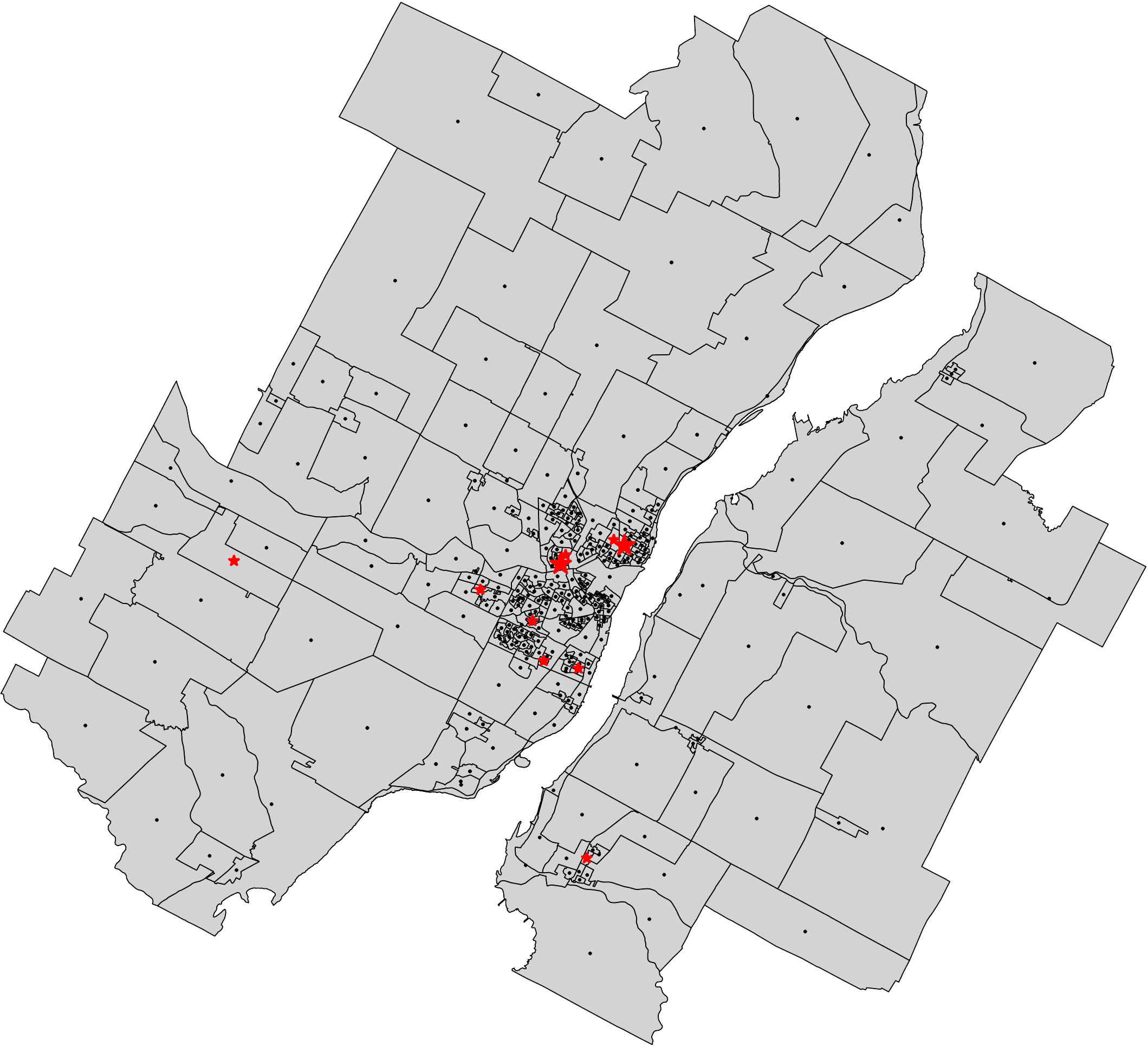}}
\caption{$\boldsymbol{\lambda}$ = \texttt{G}}
\label{fig:TR_G}
\end{subfigure}
\begin{subfigure}[b]{0.49\textwidth}
\centering
\fbox{\includegraphics[scale = 0.4]{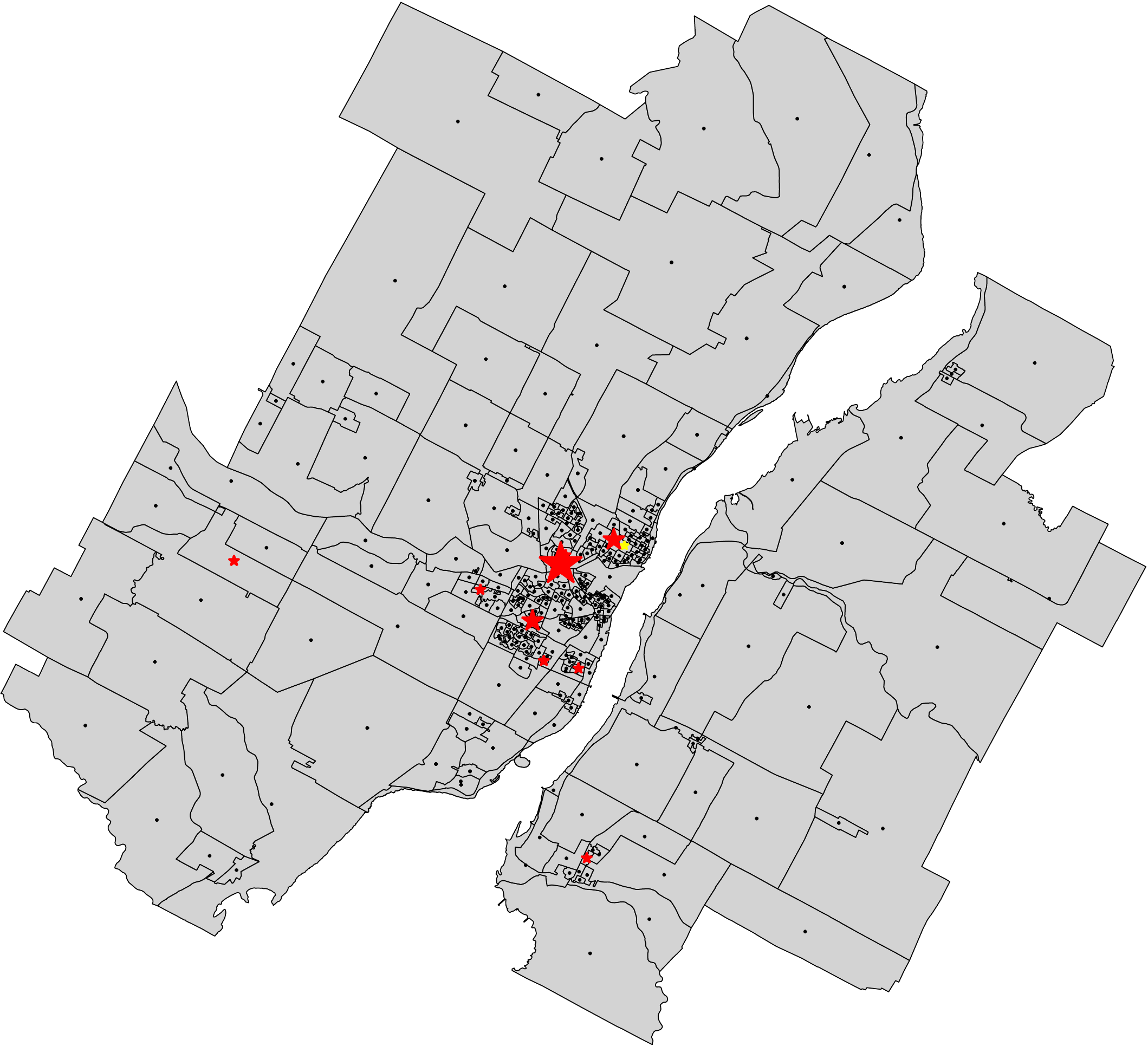}}
\caption{$\boldsymbol{\lambda}$ = \texttt{K}}
\label{fig:TR_K}
\end{subfigure}\hfill
\begin{subfigure}[b]{0.49\textwidth}
\centering
\fbox{\includegraphics[scale = 0.4]{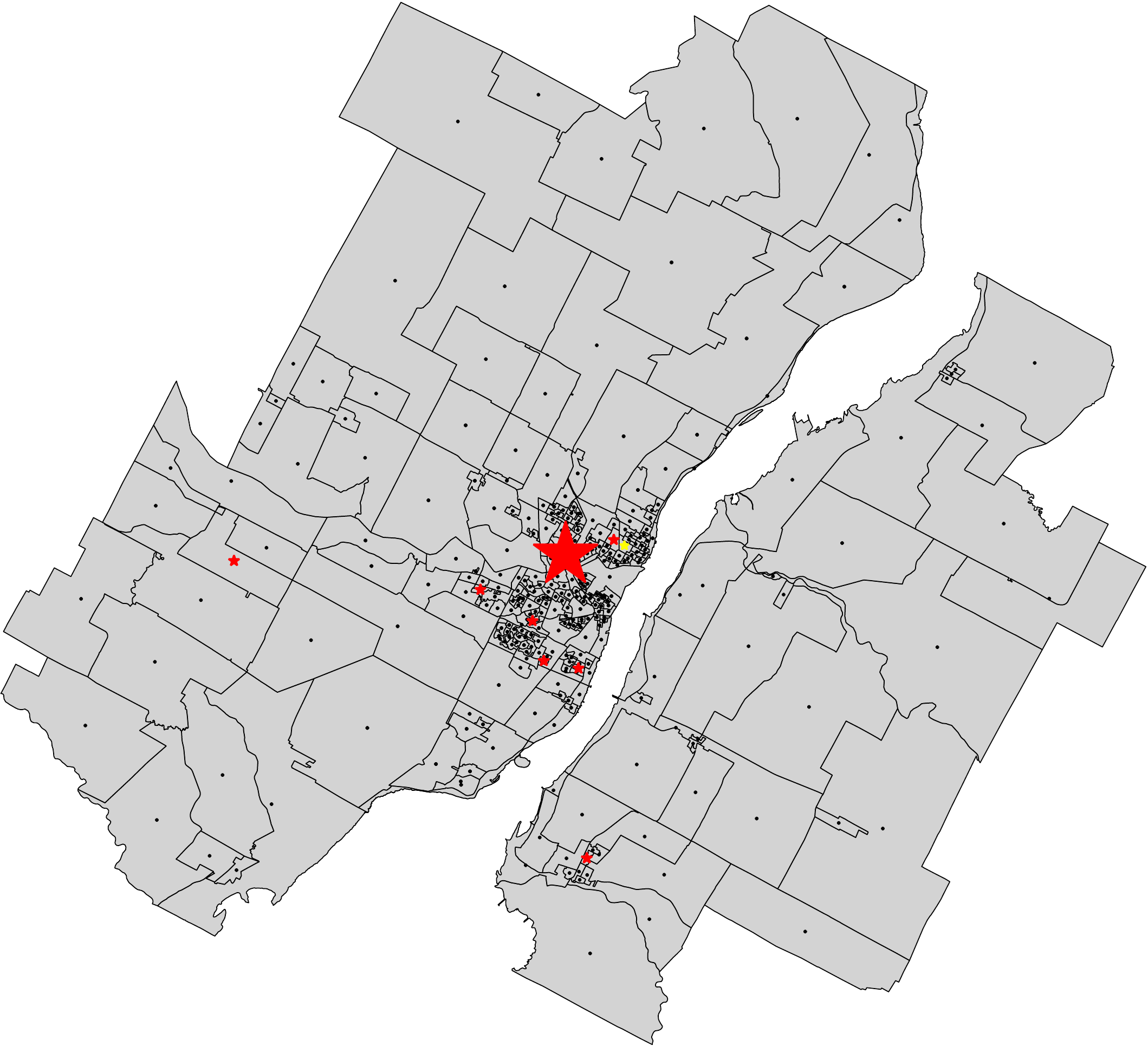}}
\caption{$\boldsymbol{\lambda}$ = \texttt{L}}
\label{fig:TR_L}
\end{subfigure}
\caption{\texttt{Real-Data} with $|\I| = 317, |\J| = 4, |\T| = 4, |\S| = 5,$ and $T = 4.5$. Example of the optimal placement of charging stations for electric vehicles in the city of Trois-Rivières, Québec, for different $\boldsymbol{\lambda}$-values. Red stars denote open stations, with the size of the star indicating the number of outlets, yellow stars represent closed stations, and black points indicate the client class for each region.}
\label{fig:ex_TR}
\end{figure}

\section{Conclusions} \label{sec:conclusiones}
In this paper, we provide a general framework for cooperative cover problems by considering a generalized version of the CMCLP where the total level of attraction is given as an OMf of the partial attractions, i.e., a weighted sum of ordered partial attractions of open facilities. We formulate the problem as a multiperiod stochastic problem with an embedded linear assignment problem characterizing the ordered median function. As a first solution approach, we present a MILP reformulation of the problem that can be solved using general-purpose solvers like Gurobi or CPLEX. We obtain a tight and compact model by deriving several sets of valid inequalities and some preprocessing techniques for particular values of the vector of weights of the OMf. Our second solution method is based on the well-known Benders Decomposition for MILPs, where we project out the assignment problem for each customer, scenario and time period. For this setting, we are able to derive an ad-hoc algorithm to include the Benders cuts in an effective and numerically more accurate manner. We test and compare all the methods by means of an extensive battery of computational experiments, and we show the variability in the location solutions for different $\boldsymbol{\lambda}$-weights in the case study proposed by \citet{anjos2022arxiv}, which deals with the placement of charging stations for electric vehicles in the city of Trois-Rivières, Québec (Canada). In view of the results obtained for the largest and more challenging instances, our exact approaches can be complemented with the development of tailored algorithms and efficient methods such as heuristics that take into account the cooperative decision rule of the customers. 

Further research on the topic includes, among others, the consideration of capacity constraints in the facilities that depend, for instance, on the type of facility installed. From a modeling point of view, the problem would be seen as a cooperative location-allocation problem. Another possible extension of the OCMCLP consists in robustifying the ordered median function problem associated to the total level of attraction. Indeed, since not only the partial attractions of the customers are uncertain, but the actual $\boldsymbol{\lambda}$-weights as well, we can consider this problem with variable $\boldsymbol{\lambda}$-weights that meet certain conditions associated to the knowledge of the customer, and optimize the total level of attraction in the worst-case. Considering non-monotone or negative $\boldsymbol{\lambda}$-weights are other extensions of the problem that may have applications in settings where the total level of attraction meets different requirements. For instance, \cite{marianov2012} discuss the optimal locations of multiple television transmitters. The transmitters are allowed to be constructive (cooperating) as well as destructive (interfering), and in the latter case the weights are non-monotone.

Finally, the study of the OCMCLP in the plane and in networks are extensions of the problem with a wide range of applications. This would be the starting point for a systematic study of ordered cooperative cover models.

\section*{Acknowledgements} 

This work was supported in part by the European Research Council (ERC) under the EU Horizon 2020 research and innovation program (grant agreement No. 755705), in part by the Spanish Ministry of Science and Innovation (AEI/10.13039/501100011033) through project PID2020-115460GB-I00, and AEI grant number RED2022-134149-T (Thematic Network: Location Science and Related Problems). 
C.\ Dom\'inguez and R.\ G\'azquez are also financially supported through the Research Program for Young Talented Reseachers of the University of M\'alaga under Project B1-2022\_37. 
Finally, the authors thankfully acknowledge the computer resources, technical expertise, and assistance provided by the SCBI (Supercomputing and Bioinformatics) center of the University of M\'alaga.

\bibliographystyle{elsarticle-harv} 
\bibliography{00_references}
	
\newpage
\appendix

\renewcommand{\thesection}{Appendix \Alph{section}}
\section{Proof of Proposition \ref{prop_linearizacion}}
In a first step towards obtaining a linear model, we start with constraint \eqref{SingleLevelNoLineal_rest_z0}. For fixed $i$, $t$, $s$ and partial attractions $u^{ts}_{ij}$, consider the feasible region for $(z^{ts}_i, \sigma^{ts}_{ijr})$  given by constraints \eqref{SingleLevelNoLineal_rest_z0}-\eqref{SingleLevelNoLineal_zcontinuas}:
\begin{align*}
W(i,t,s) := & \left\{(\boldsymbol{\sigma}_i^{ts},z_i^{ts}) \in \mathbb{R}_+^{|\J|\times |\J|} \times [0,1]: \left(T_i^{ts}-\sum_{j \in \J} \sum_{r \in \J}\lambda_{ir} u_{ij}^{ts}  \sigma_{ijr}^{ts}\right)z_i^{ts} \le 0, \right. \\
& \left. 0 \le \sum_{j \in \J} \sigma_{ijr}^{ts} \le 1\ \forall  r \in \J, \:  0 \le \sum_{r\in \J} \sigma_{ijr}^{ts} \le 1\ \forall j \in \J \right\},  
\end{align*}
\noindent where the value of $U^{ts}_i$ has been replaced in \eqref{SingleLevelNoLineal_rest_z0} using \eqref{SingleLevelNoLineal_U} and is expressed as a function of $\boldsymbol{\sigma}$.  
Defining sets $W^0$ and $W^1$ in the following manner:
\begin{align*}
W^0(i,t,s) := &  \left\{(\boldsymbol{\sigma}_i^{ts},z_i^{ts}) \in \mathbb{R}_+^{|\J|\times |\J|+1}: \boldsymbol{\sigma}_i^{ts}=\boldsymbol{0}, z_i^{ts}=0\right\}, \\
W^1(i,t,s) := & \left\{(\boldsymbol{\sigma}_i^{ts},z_i^{ts}) \in \mathbb{R}_+^{|\J|\times |\J|+1}: g(\boldsymbol{\sigma}_i^{ts}):=T_i^{ts}-\sum_{j \in \J} \sum_{r \in \J}\lambda_{ir} u_{ij}^{ts}  \sigma_{ijr}^{ts} \le 0, \right. \\
& \left. 0 \le \sum_{j \in \J} \sigma_{ijr}^{ts} \le 1\ \forall  r \in \J, \:  0 \le \sum_{r\in \J} \sigma_{ijr}^{ts} \le 1\ \forall j \in \J, z_i^{ts}=1\right\}, 
\end{align*}
\noindent  we have that $W \supset W^0 \cup W^1$, where $z_i^{ts} \in \{0,1\}$ can be viewed as an indicator variable. 
The inclusion is strict because when $z^{ts}_i=0$ the assignment is irrelevant, so we only keep the feasible assignment $\boldsymbol{\sigma}= \boldsymbol{0}$ and reject the rest of spurious solutions. On the other hand, for $z^{ts}_i=1$ the assignment $\boldsymbol{\sigma}$ needs to satisfy $g(\boldsymbol{\sigma})$ in order to be feasible for \eqref{SingleLevelNoLineal}. In fact, we define $g(\boldsymbol{\sigma})$ to emphasize that this constraint only needs to be satisfied for $z=1$, and to keep a notation consistent with \citet{gunluk2012}. Set $W^0$ is a point and $W^1$ is convex and bounded, so the convex hull of $W^0 \cup W^1$ can be characterized applying a perspective transformation \citep[see][]{gunluk2012}. Using Lemma 3.1 and Corollary 3.1 from this paper, we obtain that $conv\left(W^0\cup W^1\right) = \rm{closure}(W^-)$, with
\begin{align*}
    W^-(i,t,s) = & \left\{  (\boldsymbol{\sigma}_i^{ts},z_i^{ts}) \in \mathbb{R}_+^{|\J|\times |\J|+1}:  z^{ts}_ig(\frac{\boldsymbol{\sigma}_i^{ts}}{z^{ts}_i}) = T_i^{ts}z^{ts}_i-\sum_{j \in \J} \sum_{r \in \J}\lambda_{ir} u_{ij}^{ts}  \sigma_{ijr}^{ts} \le 0, \right. \\
    &  \left. 0 \le \sum_{j \in \J} \sigma_{ijr}^{ts} \le z_i^{ts}\ \forall  r \in \J,   0 \le \sum_{r\in \J} \sigma_{ijr}^{ts} \le z_i^{ts}\ \forall j \in \J, 0 < z^{ts}_i \le 1  \right\}. 
\end{align*}

Next, to linearize the remaining bilinear terms $u_{ij}^{ts}  \sigma_{ijr}^{ts}$ we exploit the fact that $\sigma$ are binary. Although they are stated as linear in $W^-(i,t,s)$, w.l.o.g.\ we can restrict them back to take values in \{0,1\} in order to linearize the product of a binary variable and a continuous bounded variable by a continuous variable in the manner of \cite{mccormick1976}. However, after this linearization we can no longer relax the integrality constraints on $\sigma$. Thus, defining a new set of variables $w_{ijr}^{ts} := u_{ij}^{ts}\sigma_{ijr}^{ts}$, $\forall j,r \in \J$, the linearization of constraint $T_i^{ts}z^{ts}_i \le \sum_{j \in \J} \sum_{r \in \J}\lambda_{ir} u_{ij}^{ts}  \sigma_{ijr}^{ts}$ reads:
\begin{equation*} \label{rest_zu_linealizada}  
T_i^{ts} z^{ts}_i \le \sum_{j \in \J} \sum_{r \in \J}\lambda_{ir} w_{ijr}^{ts} , \quad \forall i \in \I, t\in \T, s\in \S, 
\end{equation*}
\noindent and the following sets of constraints need to be added to the model:
\begin{subequations}
\begin{align}
 &  w_{ijr}^{ts} \le M^{ts}_{ij}\sigma_{ijr}^{ts}, \quad \forall i \in \I, t\in \T, s\in \S, j,r \in \J, \label{uuu_cotasuperiorz2}\\
 &  w_{ijr}^{ts} \le u_{ij}^{ts}, \quad \forall i \in \I, t\in \T, s\in \S, j,r \in \J,  \label{uuu_cotasuperiorw2} \\
 & w_{ijr}^{ts} \ge u_{ij}^{ts} - M^{ts}_{ij}(1-\sigma_{ijr}^{ts}), \quad \forall i \in \I, t\in \T, s\in \S, j,r \in \J,  \label{uuu_cotainferiorw2} \\
 & w_{ijr}^{ts} \ge 0, \quad \forall i \in \I, t\in \T, s\in \S, j,r \in \J,  \label{uuu_cotainferior0}
\end{align}
\end{subequations}
\noindent where for each $i \in \I$, $t\in \T$, $s\in \S$, each Big-M constant is an upper bound on the value of $u^{ts}_{ij}$, $M^{ts}_{ij}:= a_{ij|\K_j|}^{ts}$. Since we maximize on $z$ (and hence on $w$), we can omit the third set of constraints \eqref{uuu_cotainferiorw2} and still obtain a valid model. 
\section{Proof of Theorem \ref{tma:AlgoritmoBenders}} \label{app:tmaAlgoritmoBenders}
All we have to show is that $(\sigma^*)$ is feasible for the primal, $({\gamma}^*, \delta^*)$ is feasible for the dual, the optimal value of both problems is the same and complementary slackness holds. 

\textbf{Primal feasibility.} Constraints \eqref{SUB_sumasigmajk} are satisfied because, for a fixed $r\in \J$, $\sum_{j \in \J} \sum_{k \in \K_j}\sigma_{jkr} = \sigma^*_{\tau^{-1}(r)\bar{k}_{\tau^{-1}(r)}r} = 1$. Constraints \eqref{SUB_sumasigmar} hold because $\sigma^*_{jkr}=1 \Rightarrow k = \bar{k}_j$ by definition in Algorithm \ref{alg:BendersPrimal}. And constraints \eqref{SUB_sumasigmark} are satisfied because \eqref{SUB_sumasigmar} are. 

\textbf{Dual feasibility.} For each $(j,r,k)$, we prove that \eqref{SUBDUAL_sigma} are satisfied depending on $k$. First of all, we prove that 
\begin{equation*}
    \gamma^*_r + \delta^*_j \ge \lambda_r\bar{u}_j \quad \forall j,r\in\J.
\end{equation*}
We distinguish three mutually exclusive cases:
\begin{enumerate}
    \item $j = \tau(r)$. In this case, 
    \begin{multline*}
        \gamma^*_r + \delta^*_j = \gamma_r + \delta_{\tau(r)} = \dsum_{r'=r}^{|\J|-1} \left(\lambda_{r'} - \lambda_{r'+1}\right) \bar{u}_{\tau(r')} + \lambda_{|\J|}\bar{u}_{\tau(|\J|)} +  \sum_{r' = r}^{|\J|-1} \lambda_{r'+1}\left( \bar{u}_{\tau(r')} - \bar{u}_{\tau(r'+1)}\right) = \\
        = \dsum_{r'=r}^{|\J|} \lambda_{r'}\bar{u}_{\tau(r')} + \sum_{r' = r}^{|\J|-1} \lambda_{r'+1} \bar{u}_{\tau(r')} - \dsum_{r'=r}^{|\J|-1} \lambda_{r'+1}\bar{u}_{\tau(r')}  - \sum_{r' = r+1}^{|\J|} \lambda_{r'}\bar{u}_{\tau(r')} = \lambda_r\bar{u}_j.
    \end{multline*}

    \item $j = \tau(r+n)$, with $n \in \mathbb{N}^+$. 
    \begin{multline*}
        \gamma^*_r + \delta^*_j = \dsum_{r'=r}^{|\J|} \lambda_{r'}\bar{u}_{\tau(r')} - \dsum_{r'=r}^{|\J|-1} \lambda_{r'+1}\bar{u}_{\tau(r')} + \sum_{r' = r+n}^{|\J|-1} \lambda_{r'+1} \bar{u}_{\tau(r')} - \sum_{r' = r+n+1}^{|\J|} \lambda_{r'}\bar{u}_{\tau(r')} = \\
        = \sum_{r'=r}^{r+n-1} (\lambda_{r'}-\lambda_{r'+1})\bar{u}_{\tau(r')} + \lambda_{r+n}\bar{u}_{\tau(r+n)} \ge \sum_{r'=r}^{r+n-2} (\lambda_{r'}-\lambda_{r'+1})\bar{u}_{\tau(r')} + \lambda_{r+n-1}\bar{u}_{\tau(r+n)} \ge \dots \ge \lambda_r\bar{u}_{\tau(r+n)},
    \end{multline*}
\noindent where we apply several times the inequality  $\lambda_r\bar{u}_{\tau(r)}  + \lambda_{r+1}\bar{u}_{\tau(r+1)} \ge \lambda_{r+1}\bar{u}_{\tau(r)} + \lambda_r\bar{u}_{\tau(r+1)}$, which holds because 
$(\lambda_r-\lambda_{r+1})(\bar{u}_{\tau(r)}-\bar{u}_{\tau(r+1)}) \ge 0$ by definition of $\lambda$ and $\tau$.

    \item $j = \tau(r-n)$, with $n \in \mathbb{N}^+$.
    \begin{multline*}
        \gamma^*_r + \delta^*_j = \dsum_{r'=r}^{|\J|} \lambda_{r'}\bar{u}_{\tau(r')} - \dsum_{r'=r}^{|\J|-1} \lambda_{r'+1}\bar{u}_{\tau(r')} + \sum_{r' = r-n}^{|\J|-1} \lambda_{r'+1} \bar{u}_{\tau(r')} - \sum_{r' = r-n+1}^{|\J|} \lambda_{r'}\bar{u}_{\tau(r')} = \\
        = \sum_{r' = r-n}^{r-1} \lambda_{r'+1} \bar{u}_{\tau(r')} - \sum_{r' = r-n}^{r-2} \lambda_{r'+1}\bar{u}_{\tau(r'+1)} = 
        \sum_{r' = r-n}^{r-3} \lambda_{r'+1} (\bar{u}_{\tau(r')} -\bar{u}_{\tau(r'+1)}) + \lambda_r\bar{u}_{\tau(r-2)} \ge \dots \ge \lambda_r\bar{u}_{\tau(r-n)}
        \end{multline*}
\noindent where we apply several times the inequality  $\lambda_r\bar{u}_{\tau(r-1)}  + \lambda_{r-1}\bar{u}_{\tau(r-2)} \ge \lambda_{r-1}\bar{u}_{\tau(r-1)} + \lambda_r\bar{u}_{\tau(r-2)}$, which holds because 
$(\lambda_{r-1}-\lambda_r)(\bar{u}_{\tau(r-2)}-\bar{u}_{\tau(r-1)}) \ge 0$ by definition of $\lambda$ and $\tau$.
\end{enumerate}

This reasoning proves that \eqref{SUBDUAL_sigma} hold $\forall j,r\in \J$, $\forall k\le \bar{k}_j$. To finish, it suffices to prove that for $k > \bar{k}_j$, $\eta^*_{jk} = \max_{r\in \J} \left\{\lambda_ra_{jk} - \gamma^*_r - \delta^*_j \right\} \ge 0$, i.e., that $\max_{r\in \J} \left\{\lambda_ra_{jk} - \gamma^*_r - \delta^*_j \right\}$ attains its maximum value when $r=r^*$ defined in Algorithm \ref{alg:BendersDual}.

Since $\delta^*_j$ does not depend on $r$, let us define $W(r) := \left\{\lambda_ra_{jk} - \gamma^*_r\right\}$ and study the variation of $W(r)$ when $r$ increases:
\begin{equation}
    W(r+1)-W(r) = (\lambda_r-\lambda_{r+1})(\bar{u}_{\tau(r)} - a_{jk}).
\end{equation}
Given that $\lambda_r-\lambda_{r+1} \ge 0$ $\forall r$, $W(r+1)-W(r) \ge 0$ if and only if $\bar{u}_{\tau(r)} - a_{jk}\ge 0$. When $r$ increases, $a_{jk}$ is fixed and $\bar{u}_{\tau(r)}$ decreases. Therefore, $W(r)$ increases while $\bar{u}_{\tau(r)} -a_{jk} \ge 0$ and then decreases, so the maximum is obtained for $r=r^*_{jk}$.

\textbf{Strong Duality.}
The primal optimal value is $\sum_{j \in \J}\sum_{r \in \J} \sum_{k \in \K_j}\lambda_r a_{jk} \sigma^*_{jkr} =$ $\sum_{j \in \J} \sigma^*_{j\bar{k}_j\tau^{-1}(j)} =$ $\sum_{r\in \J} \lambda_r\bar{u}_{\tau(r)}$. The dual optimal value is $\sum_{r \in \J} \gamma^*_r + \sum_{j\in \J} \delta^*_j + \sum_{j\in \J} \sum_{k \in \K_j} x_{jk} \eta^*_{jk} = \sum_{r \in \J} \gamma^*_r + \sum_{j\in \J} \delta^*_j = \sum_{r\in \J} (\gamma^*_r + \delta^*_{\tau(r)}) = \sum_{r\in \J} \lambda_r\bar{u}_{\tau(r)}$, as proven in Case 1.\ of \textbf{Dual Feasibility}.

\section{Additional results from the synthetic data set} \label{app:syn}

\begin{table}[H]
\begin{center}
\small
    \begin{tabular}{ccccrrrrrrrrr}
    \toprule
$|\S|$ & $|\I|$ & $\boldsymbol{\lambda}$ & $|\J|$ & \multicolumn{3}{c}{\texttt{Time[s] (\#Solved)}} &  \multicolumn{3}{c}{\texttt{MIPGap[\%]}} & \multicolumn{3}{c}{\texttt{FGap[\%]}} \\
\cmidrule(lr){5-7}\cmidrule(lr){8-10}\cmidrule(lr){11-13}
 &  &  & \texttt{SL} & \texttt{VI} & \texttt{B} & \texttt{SL} & \texttt{VI} & \texttt{B} & \texttt{SL} & \texttt{VI} & \texttt{B} \\
 \hline
5 & 20 & \texttt{C} & 10 & 0.2 (30) & 3.4 (30) & 1.3 (30) & 0.0 & 0.0 & 0.0 & 0.0 & 0.0 & 0.0 \\
     & & \texttt{G} & 10 & 655.6 (20) & 267 (28) & 46.1 (30) & 34.0 & 9.0 & 0.0 & 30.2 & 8.6 & 0.0 \\
     & & \texttt{K} & 10 & 402.0 (26) & 36.3 (30) & 6.5 (30) & 22.6 & 0.0 & 0.0 & 22.4 & 0.0 & 0.0 \\
     & & \texttt{L} & 10 & 270.0 (27) & 100.2 (30) & 22.7 (30) & 5.3 & 0.0 & 0.0 & 5.3 & 0.0 & 0.0 \\
& 40 & \texttt{C} & 10 & 0.7 (30) & 6.2 (30) & 3.1 (30) & 0.0 & 0.0 & 0.0 & 0.0 & 0.0 & 0.0 \\
             & &  & 30 & 5.3 (30) & 58.4 (25) & 40.3 (27) & 0.0 & 29.7 & 5.0 & 0.0 & 29.7 & 5.0 \\
   & & \texttt{G} & 10 & 2709.1 (6) & 746.6 (18) & 376.7 (27) & 50.9 & 29.0 & 6.2 & 47.9 & 29.0 & 6.2 \\
             & &  & 30 & \texttt{TL} (0) & \texttt{TL} (0) & 2109.2 (3) & 114.2 & 109.2 & 64.7 & 96.0 & 45.6 & 64.7 \\
   & & \texttt{K} & 10 & 1346.6 (14) & 144.8 (30) & 62.7 (29) & 12.9 &  & 14.9 & 12.9 &  & 14.9 \\
             & &  & 30 & \texttt{TL} (0) & 2013.9 (6) & 376 (20) & 40.1 & 18.4 & 12.8 & 35.8 & 18.2 & 12.1 \\
   & & \texttt{L} & 10 & 1511.5 (15) & 519.7 (26) & 139.5 (30) & 26.1 & 18.5 & 0.0 & 24.1 & 18.5 & 0.0 \\
             & &  & 30 & \texttt{TL} (0) & \texttt{TL} (0) & 838.5 (9) & 79.7 & 51.1 & 22.6 & 70.8 & 49.5 & 21.0 \\
10 & 20 & \texttt{C} & 10 & 0.6 (30) & 6.8 (30) & 3.3 (30) & 0.0 & 0.0 & 0.0 & 0.0 & 0.0 & 0.0 \\
   & & \texttt{G} & 10 & 2561.3 (7) & 889.8 (20) & 357.2 (27) & 43.7 & 24.1 & 2.7 & 43.7 & 24.1 & 2.7 \\
   & & \texttt{K} & 10 & 1267.2 (15) & 138.9 (30) & 35.3 (29) & 5.8 & 0.0 & 8.8 & 5.8 & 0.0 & 8.8 \\
   & & \texttt{L} & 10 & 1108.5 (17) & 327.8 (25) & 105.3 (30) & 14.4 & 23.8 & 0.0 & 14.4 & 21.3 & 0.0 \\
& 40 & \texttt{C} & 10 & 1.3 (30) & 19.5 (30) & 9.8 (30) & 0.0 & 0.0 & 0.0 & 0.0 & 0.0 & 0.0 \\
             & &  & 30 & 11.1 (30) & 241.2 (22) & 113.2 (24) & 0.0 & 35.6 & 6.6 & 0.0 & 35.6 & 6.6 \\
 & & \texttt{G} & 10 & \texttt{TL} (0) & 2074.5 (11) & 685.8 (13) & 57.0 & 33.9 & 16.8 & 54.1 & 33.9 & 16.8 \\
 & &  & 30 & \texttt{TL} (0) & \texttt{TL} (0) & \texttt{TL} (0) & 151.7 & 183.1 & 56.7 & 104.2 & 45.7 & 56.7 \\
 & & \texttt{K} & 10 & 2306.3 (2) & 498.7 (24) & 173.5 (24) & 8.5 & 5.6 & 8.5 & 8.5 & 5.6 & 8.5 \\
 & &  & 30 & \texttt{TL} (0) & \texttt{TL} (0) & 661.8 (15) & 46.1 & 21.0 & 10.9 & 38.1 & 17.3 & 10.9 \\
 & & \texttt{L} & 10 & 2741.8 (6) & 1324.3 (18) & 676.6 (22) & 28.2 & 34.5 & 8.3 & 28.2 & 33.8 & 8.3 \\
 & &  & 30 & \texttt{TL} (0) & \texttt{TL} (0) & 2896.1 (6) & 104.1 & 69.9 & 37.4 & 78.2 & 28.3 & 37.4 \\
    \bottomrule
\end{tabular}
\caption{\texttt{Syn-Data}. The total number of instances per row is 30. Time limit equal to 3600 seconds. Median computational time of solved instances in seconds (\texttt{Time[s]}), number of solved instances (\texttt{\#Solved}), median Gap between incumbent solution and bound (\texttt{MIPGap}), and median Gap between best solution found among the models and bound (\texttt{FGap}).} 
\label{t:sce5}
\end{center}
\end{table}

\section{Additional results from the case study} \label{app:real}

\begin{table}[H]
\begin{center}
    \begin{tabular}{ccrrrr}
    \toprule
$|\J|$ & $|\I|$ & \multicolumn{2}{c}{$T = 4.5$} & \multicolumn{2}{c}{$T = 9$} \\ \cmidrule(lr){3-4}\cmidrule(lr){5-6}
 &  & \texttt{Time[s]} \texttt{(\# Solved)} & \texttt{MIPGap[\%]} & \texttt{Time[s]} \texttt{(\# Solved)} & \texttt{MIPGap[\%]} \\
\hline
10 & 100 & 1687.9 (79) & 0.4 & 6.3 (22) & 1.5 \\
 & 200 & 50.8 (20) & 1.0 & 18.2 (20) & 3.2 \\
 & 317 & 114.2 (20) & 1.4 & 34.3 (20) & 4.0 \\
20 & 100 & 55.2 (20) & 2.4 & 25.8 (20) & 8.0 \\
 & 200 & 545.1 (20) & 4.0 & 107.6 (20) & 9.8 \\
 & 317 & 7691.2 (18) & 4.4 & 243.2 (20) & 10.1 \\
30 & 100 & 156.6 (20) & 3.7 & 74.9 (20) & 10.8 \\
 & 200 & 17463.3 (11) & 5.4 & 1341.4 (20) & 13.3 \\
 & 317 & \texttt{TL} (0) & 6.0 & 10155.2 (18) & 14.2 \\
    \bottomrule
\end{tabular}
\caption{\texttt{Real-Data} with $|\T| = 10$. The total number of instances per row is 80. Time limit equal to 28800 seconds (8 hours). Median computational time of solved instances in seconds (\texttt{Time[s]}), number of solved instances (\texttt{\#Solved}), and median Gap between the incumbent solution and the best bound (\texttt{MIPGap}).} 
\label{t:anjos-computational2}
\end{center}
\end{table}

\begin{table}[H]
\begin{center}
    \begin{tabular}{ccrrrrrr}
    \toprule
$|\J|$ & $|\I|$ & \multicolumn{3}{c}{$T = 4.5$} & \multicolumn{3}{c}{$T = 9$} \\ \cmidrule(lr){3-5}\cmidrule(lr){6-8}
 &  & \multicolumn{1}{c}{\texttt{G}} & \multicolumn{1}{c}{\texttt{K}} & \multicolumn{1}{c}{\texttt{L}} & \multicolumn{1}{c}{\texttt{G}} & \multicolumn{1}{c}{\texttt{K}} & \multicolumn{1}{c}{\texttt{L}} \\
 \hline
10 & 100 & 0.7 & 0.7 & 0.7 & 1.4 & 1.4 & 1.5 \\
   & 200 & 0.5 & 0.5 & 0.5 & 0.8 & 0.8 & 0.8 \\
   & 317 & 0.3 & 0.4 & 0.4 & 0.7 & 0.7 & 0.7 \\
20 & 100 & 0.7 & 1.0 & 0.8 & 1.9 & 2.7 & 2.2 \\
   & 200 & 0.4 & 0.6 & 0.5 & 0.7 & 1.3 & 0.8 \\
   & 317 & 0.3 & 0.4 & 0.3 & 0.4 & 0.9 & 0.7 \\
30 & 100 & 0.5 & 1.0 & 0.7 & 1.2 & 2.6 & 1.5 \\
   & 200 & 0.2 & 0.6 & 0.3 & 0.3 & 1.0 & 0.5 \\
   & 317 &  -- & --  & -- & 0.1 & 0.6 & 0.4 \\
    \bottomrule
\end{tabular}
\caption{\texttt{Real-Data} with $|\T| = 10$. The total number of instances per row is 20. Time limit equal to 28800 seconds (8 hours). Average of the regret between \texttt{C} and other $\boldsymbol{\lambda}$-values.} 
\label{t:deviation_t10}
\end{center}
\end{table}

\end{document}